\documentclass[reqno,11pt]{article}
\usepackage{mathrsfs}
\usepackage{amssymb}
 \UseRawInputEncoding
\usepackage{amsthm}
\usepackage{amsmath}
\usepackage{amsfonts}
\usepackage{enumerate}
\usepackage{bm}

\usepackage{multicol}

\usepackage{subfigure}
\usepackage{graphicx}
\usepackage{setspace,color,wrapfig}
\usepackage[colorlinks,linkcolor=blue,citecolor=cyan]{hyperref}

\numberwithin{equation}{section}
\newtheorem{theorem}{Theorem}[section]
\newtheorem{lemma}[theorem]{Lemma}

\newtheorem{problem}[theorem]{Problem}
\theoremstyle{definition}
\newtheorem{definition}[theorem]{Definition}

\theoremstyle{remark}
\newtheorem{remark}[theorem]{Remark}

\begin{document}
\title{  Subsonic  flows with a contact discontinuity in  a  finitely long axisymmetric cylinder  }
\author{Shangkun Weng\thanks{School of mathematics and statistics, Wuhan University, Wuhan, Hubei Province, 430072, People's Republic of China. Email: skweng@whu.edu.cn}\and Zihao Zhang\thanks{School of mathematics and statistics, Wuhan University, Wuhan, Hubei Province, 430072, People's Republic of China. Email: zhangzihao@whu.edu.cn}}
\date{}
\maketitle
\newcommand{\de}{{\mathrm{d}}}
\def\div{{\rm div\,}}
\def\curl{{\rm curl\,}}
\def\th{\theta}
\newcommand{\ro}{{\rm rot}}
\newcommand{\sr}{{\rm supp}}
\newcommand{\sa}{{\rm sup}}
\newcommand{\va}{{\varphi}}
\newcommand{\me}{\mathcal{M}}
\newcommand{\ml}{\mathcal{V}}
\newcommand{\md}{\mathcal{D}}
\newcommand{\mg}{\mathcal{G}}
\newcommand{\mh}{\mathcal{H}}
\newcommand{\mf}{\mathcal{F}}
\newcommand{\ms}{\mathcal{S}}
\newcommand{\mt}{\mathcal{T}}
\newcommand{\mn}{\mathcal{N}}
\newcommand{\mb}{\mathcal{P}}
\newcommand{\mm}{\mathcal{B}}
\newcommand{\mj}{\mathcal{J}}
\newcommand{\mk}{\mathcal{K}}
\newcommand{\my}{\mathcal{U}}
\newcommand{\mw}{\mathcal{W}}
\newcommand{\mq}{\mathcal{Q}}
\newcommand{\ma}{\mathcal{L}}
\newcommand{\mc}{\mathcal{C}}
\newcommand{\mi}{\mathcal{I}}
\newcommand{\n}{\nabla}
\newcommand{\e}{\tilde}
\newcommand{\h}{\hat}
\newcommand{\m}{\omega}
 \newcommand{\q}{{\rm R}}
\newcommand{\p}{{\partial}}
\newcommand{\z}{{\varepsilon}}
\renewcommand\figurename{\scriptsize Fig}
\begin{abstract}
 This paper concerns the structural stability of subsonic flows with a contact discontinuity  in  a  finitely long axisymmetric cylinder. We establish the existence and uniqueness of axisymmetric subsonic flows with a contact discontinuity by prescribing the horizontal mass flux distribution, the swirl velocity, the  entropy
  and the Bernoulli's quantity  at the entrance and the radial velocity at the exit. It can be formulated as a free boundary problem with the contact discontinuity  to be determined simultaneously with the flows.  Compared with the two-dimensional case, a new difficulty arises due to  the  singularity near the axis. One of the key points in the analysis is the introduction of an  invertible modified Lagrangian transformation which can  overcome this difficulty and straighten the contact discontinuity. Another one is
    to utilize the  deformation-curl decomposition for the steady Euler system introduced in \cite{WX19} to effectively decouple the  hyperbolic and elliptic  modes. Finally,  the contact discontinuity will be  located by using the implicit function theorem.
\end{abstract}
\begin{center}
\begin{minipage}{5.5in}
\end{minipage}
\end{center}
\section{Introduction  }\noindent
\par In this paper, we are concerned with the structural stability of  subsonic flows with a  contact discontinuity governed by the three-dimensional steady  full Euler system in  a  finitely long axisymmetric cylinder.  The  three-dimensional steady  full Euler system for compressible  inviscid gas  is of the  following form:
\begin{align}\label{1-1}
\begin{cases}
\p_{x_1}(\rho u_1)+\p_{x_2}(\rho u_2)+\p_{x_3}(\rho u_3)=0,\\
\p_{x_1}(\rho u_1^2) + \p_{x_2}(\rho u_1 u_2) + \p_{x_3}(\rho u_1u_3)+\p_{x_1}P=0,\\
\p_{x_1}(\rho u_1 u_2) + \p_{x_2}(\rho u_2^2) + \p_{x_3}(\rho u_2u_3)+\p_{x_2}P=0,\\
\p_{x_1}(\rho u_1 u_3) + \p_{x_2}(\rho u_2 u_3) + \p_{x_3}(\rho u_3^2)+\p_{x_3}P=0,\\
\p_{x_1}(\rho u_1 E+u_1 P)+\p_{x_2}(\rho u_2 E+u_2 P)+\p_{x_3}(\rho u_3 E+u_3 P)=0,\\
\end{cases}
\end{align}
where $\bm u=(u_1,u_2,u_3)$ is the velocity, $\rho$ is the density, $P$ is the pressure, $E$ is the
 energy, respectively.
 For polytropic gas, the equation of state and the
energy are of the form
\begin{equation*}
P= A(S)\rho^{\gamma}, \quad{\rm {and}}\quad E=\frac{1}{2}|{\bm u}^2|+\frac{ P}{(\gamma-1)\rho},
\end{equation*}
where  $A(S)= R e^{S}$ and $\gamma\in (1,+\infty)$, $R$ are positive constants.
   Denote the Bernoulli's function and the local sonic speed by $B=\frac12|{\textbf{u}}|^2+\frac{\gamma P}{(\gamma-1)\rho}$  and $ c(\rho,A)=\sqrt{A\gamma} \rho^{\frac{\gamma-1}{2}}$, respectively. Then  the system \eqref{1-1} is hyperbolic for supersonic flows ($ |\textbf{u}|>c(\rho,A) $), and hyperbolic-elliptic coupled for subsonic flows ($ |\textbf{u}|<c(\rho,A) $).
\par  To understand   the contact discontinuity surface, we first  give the definition of steady flows  with  a contact discontinuity.
 Let $ \md\in \mathbb{R}^3 $ be an open and connected domain. Suppose that a non-self-intersecting $ C^1$-curve  $ \Gamma$ divides $ \md $ into two disjoint open subsets $ \md^\pm$  such that $ \md= \md^- \cup\Gamma \cup \md^+ $.
 Assume that $ \bm U = (\rho, u_1, u_2,u_3,P)$ satisfies the following properties:
\begin{enumerate}[(1)]
\item $\bm U \in[L^\infty(\md)\cap C^{1}_{loc}(\md^\pm)\cap C^0_{loc}(\md^\pm\cup \Gamma)]^5$;
\item For any $\eta\in C_0^{\infty}(\md)$,
\begin{equation}\label{1-2}
\begin{cases}
\int_{\md}(\rho u_1\p_{x_1}\eta+\rho u_2\p_{x_2}\eta+\rho u_3\p_{x_3}\eta)\de \mathbf{x}=0,\\
\int_{\md}((\rho u_1^2+P)\p_{x_1}\eta+\rho u_1u_2\p_{x_2}\eta+\rho u_1u_3\p_{x_3}\eta)\de \mathbf{x}=0,\\
\int_{\md}(\rho u_1u_2\p_{x_1}\eta+(\rho u_2^2+P)\p_{x_2}\eta+\rho u_2u_3\p_{x_3}\eta)\de \mathbf{x}=0,\\
\int_{\md}(\rho u_1u_3\p_{x_1}\eta+\rho u_2u_3\p_{x_2}\eta+ (\rho u_3^2+P)\p_{x_3}\eta)\de \mathbf{x}=0,\\
\int_{\md}(\rho u_1(E+\frac{P}{\rho})\p_{x_1}\eta+\rho u_2(E+\frac{P}{\rho})\p_{x_2}\eta+\rho u_3(E+\frac{P}{\rho})\p_{x_3}\eta)\de \mathbf{x}=0.\\
\end{cases}
\end{equation}
\end{enumerate}
\par By integration by parts, we get the Rankine-Hugoniot conditions:
\begin{equation}\label{1-3}
\begin{cases}
n_1[\rho u_1]+n_2[\rho u_2]+n_2[\rho u_3]=0,\\
n_1[\rho u_1^2]+n_1[P]+n_2[\rho u_1u_2]+n_3[\rho u_1u_3]=0,\\
n_1[\rho u_1u_2]+n_2[\rho u_2^2]+n_2[P]+n_3[\rho u_2u_3]=0,\\
n_1[\rho u_1u_3]+n_2[\rho u_2u_3]+n_3[\rho u_3^2]+n_3[P]=0,\\
n_1[ \rho u_1(E+\frac{P}{\rho})]+n_2[\rho u_2( E+\frac{P}{\rho})]+n_3[\rho u_3( E+\frac{P}{\rho})]=0,\\
\end{cases}
\end{equation}
where $ \mathbf{n} = (n_1, n_2,n_3) $  is the unit normal vector to $\Gamma$, and $[F](\mathbf{x}) = F_+(\mathbf{x})- F_-(\mathbf{x})$
denotes the jump across the   surface $\Gamma$ for a piecewise smooth function $ F $.
\par Let $ \bm \tau_1= (\tau_{11},\tau_{21},\tau_{31})$ and $ \bm \tau_2= (\tau_{12},\tau_{22},\tau_{32})$ as the unit tangential vectors to $\Gamma$, which means that $ \mathbf{n}\cdot{\bm \tau_k}= 0$. Taking the dot product of $(\eqref{1-3}_2,\eqref{1-3}_3,\eqref{1-3}_4)$ with $ \mathbf{n}$  and $ \bm \tau_k $ respectively, one has
\begin{equation}\label{1-4}
 [\rho(\mathbf{u}\cdot\mathbf{n})^2+P]_{\Gamma}= 0, \quad \rho(\mathbf{u}\cdot\mathbf{n})[\mathbf{u}\cdot \bm \tau_k]_{\Gamma}= 0, \quad
 {\rm{for}} \quad k=1,2.
\end{equation}
 Assume that $ \rho> 0 $  in $ \bar \md $, \eqref{1-4} implies either $ \mathbf{u}\cdot\mathbf{n} = 0 $
on $\Gamma$  or $[\mathbf{u}\cdot \bm \tau_k]_{\Gamma}=0$.
If $ \mathbf{u}\cdot\mathbf{n} \neq 0 $ and $[\mathbf{u}\cdot \bm \tau_k]_{\Gamma}=0 $ hold on $ \Gamma $, the surface $ \Gamma $ is called a shock; if the flow moves along both sides of $ \Gamma $ such that  $\mathbf{u}\cdot\mathbf{n} = 0 $ on $ \Gamma $, the  surface $ \Gamma $ is called a contact discontinuity. In the latter case, $\mathbf{u}\cdot\mathbf{n} = 0 $ and the first equation in \eqref{1-4} give $ [P]=0 $. Then we get the R-H conditions corresponding to a contact discontinuity  as follows:
\begin{equation}\label{1-5}
 \mathbf{u}\cdot\mathbf{n}=0 \quad {\rm{and}} \quad [P]=0 \quad  {\rm{on}} \ \Gamma.
 \end{equation}
 \begin{definition}
We define $ \bm U = (\rho, u_1, u_2,u_3,P)$  to be a weak solution of the full Euler system \eqref{1-1} in $ \md $ with a contact discontinuity $\Gamma$ if the the following properties hold:
\begin{enumerate}[(i)]
\item $\Gamma$ is a non-self-intersecting $ C^1$-curve  dividing $\md $ into two disjoint open subsets $ \md^\pm$  such that $ \md = \md^- \cup\Gamma \cup \md^+ $;
\item  $ \bm U $ satisfies $\rm{(1)}$ and $ \rm{(2)} $;
\item $ \rho>0 $ in $ \bar \md$;
 \item $(\mathbf{u}|_{\bar \md_-\cap\Gamma}-\mathbf{u}|_{\bar \md_+\cap\Gamma })\neq \mathbf{0} $
 holds for all $ \mathbf{x}\in\Gamma$;
 \item $ \mathbf{u}\cdot\mathbf{n}=0 $ and $ [P]=0 $ on $\Gamma$.
 \end{enumerate}
 \end{definition}
 \par This is a continuous work on the study of  subsonic  Euler flows with a contact discontinuity  in  a finitely long nozzle. In the previous work \cite{WZ23}, we established the existence and
uniqueness of  subsonic flows with a contact discontinuity in a two-dimensional  finitely long slightly curved nozzle. As an attempt to extend the results from \cite{WZ23} to the three dimensional case, this paper investigates  the structural stability of subsonic flows with a contact discontinuity  in  a finitely long axisymmetric cylinder under the suitable axisymmetric perturbations of boundary conditions.
\par The study on the steady compressible  flows  with a contact discontinuity is not only of fundamental importance in developing the mathematical theory of partial differential equations arising from fluid dynamics, but also has  important applications to engineering designs,  such as rocket launching, sharp charged jet and so on. Up to now, there have been many works in the literature on  the steady flows with a contact discontinuity. For the subsonic flow, the stability  of flat contact discontinuity in infinite nozzles was established in  \cite{BM09} and \cite{ BP19,PB19}. The authors in \cite{ BP19,PB19}   decompose the Rankine-Hugoniot conditions on the contact discontinuity via  Helmholtz decomposition so that the compactness of approximated solutions can be achieved.  The uniqueness and existence of the contact discontinuity in infinitely  long nozzles, which is not a perturbation of
the flat contact discontinuity, was obtained in \cite{CHWX19}. The existence, uniqueness and stability of subsonic flows past
an airfoil with a vortex line were obtained  in \cite{CXZ22}.    The key idea in \cite{CXZ22}  is to  use  the implicit function theorem as the framework to solve the problem of subsonic flows past
an airfoil with a vortex line.   Inspired by \cite{CXZ22}, the stability of contact discontinuity in a finite nozzle was established  in \cite{WZ23} by using the implicit function
theorem.   For the supersonic flow, the stability  of flat contact discontinuity in  finitely long nozzles was studied  in \cite{HFWX19}. The stability of three-dimensional supersonic  contact discontinuity was investigated in  \cite{WY13,WF15}. Recently, the stability  of   supersonic contact discontinuity for the two-dimensional steady rotating Euler system in a finitely curved nozzle has been  established in \cite{WSZ23}. For the transonic flow, the stability  of flat contact discontinuity in  finitely long nozzles was established in \cite{HFWX21}. The stability  of two-dimensional  transonic contact discontinuity over a solid wedge  and three-dimensional   transonic contact discontinuity were established in \cite{CYK13,CKY13} and \cite{WY15}.
\par We make some comments on the new ingredients in our analysis for the contact discontinuity problem. Note that  the contact discontinuity is part of the  solution  and is unknown, thus  this is a free boundary that separates both subsonic flows in the  inner and outer  layers of the cylinder.  The strategy to overcome this difficulty in the  two dimensional case \cite{WZ23} is to introduce a Lagrangian transformation to straighten the contact discontinuity. The idea also applies to three dimensional steady axisymmetric Euler system.  However, in the three dimensional axisymmetric setting, there is a singular term $r$ in the density equation.   Inspired by \cite{WXX21},  the singular term $ r $ in the density equation is of order $ O(r) $ near the axis $ r=0 $, hence we can  find a simple  modified
Lagrangian transformation  such that it is invertible near the axis and also straightens the contact discontinuity. Another key issue is to decompose the hyperbolic and elliptic  modes in the
steady axisymmetric Euler system. It is well-known that   the
steady axisymmetric  Euler system  is hyperbolic-elliptic mixed in subsonic regions, whose effective decomposition of elliptic and hyperbolic modes is crucial for  developing a well-defined iteration.   Here we will use the deformation-curl decomposition introduced in \cite{WX19,WS19} to effectively decouple the hyperbolic and elliptic modes in subsonic
regions.
\par The other key ingredient in our  analysis is to employ the implicit function
theorem to locate the contact discontinuity.  The idea is inspired  by the discussion of the
airfoil problem in \cite{CXZ22}. We   choose    a suitable  H\"{o}lder space   and   design a  proper  map   to verify the conditions
in the implicit function theorem.  However, it seems quite difficult   to verify that the isomorphism of the  differential of the map for general background flows with a straight contact discontinuity. Here we choose the background outer-layer flow is stagnant and restrict the perturbation only on the entrance of the inner-outer flow. In this case, the outer-layer flow is fixed and one can prove the isomorphism,   the contact discontinuity can be located by   the implicit function theorem.
\par This paper will be arranged as follows. In Section 2, we formulate the problem of subsonic flows with a contact discontinuity  in  a finitely long axisymmetric cylinder and state the main result. In Section 3,  the modified Lagrange transformation is  employed to straighten the
contact discontinuity and reformulate the free boundary value problem 2.2. Then we
 use the deformation-curl decomposition in \cite{WX19,WS19} to derive an equivalent system.
Finally, we state the main steps to solve the free boundary problem 3.1.
In Section 4, we first linearize the nonlinear  system and solve the linear  system in a suitable weighted H$\ddot{\rm{o}}$lder space. Then  the framework of the contraction mapping theorem can be used  to find the solution of the nonlinear  system.
 In Section 5, we choose a suitable H$\ddot{\rm{o}}$lder space and design a proper  map to verify the conditions in the implicit function theorem. Then by using the   implicit function theorem, we  locate the contact discontinuity. In Section 6, we finish the proof of the main theorem.
 \section{Mathematical formulation of the problem}\noindent
  \par In this section,  we first construct a special class of subsonic Euler  flows with a straight contact discontinuity in a finitely long axisymmetric cylinder. Then we give a detailed formulation  of the stability problem for these background flows with a contact discontinuity  and state the main result.
  \subsection{The background solutions}\noindent
  \par The axisymmetric  cylinder (Fig 1) of the length $ L $ is given by
\begin{equation*}
\mn:=\{(x_1,x_2,x_3)\in \mathbb{R}^3:0<x_1<L,\ 0\leq x_2^2+x_3^2<1\}.
\end{equation*}
\begin{figure}
  \centering
  \includegraphics[width=11cm,height=5cm]{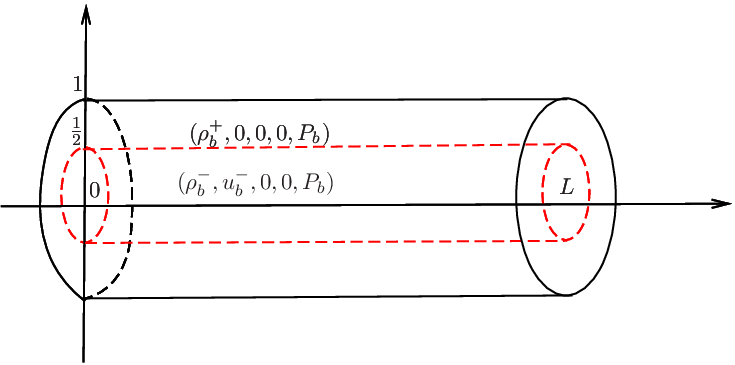}
  \caption{Subsonic flows with a contact discontinuity in an  axisymmetric cylinder}
\end{figure}
  Consider two layers  of steady smooth Euler flows  separated by  the cylindrical surface $ r=\sqrt{x_2^2+x_3^2}=\frac{1}{2} $ satisfying the following properties:
     \begin{enumerate}[(i)]
\item The velocity and density of the outer and inner layers are given by $ (0,0,0),\rho_b^{+} $ and $ (u_b^{-},0,0),\rho_b^{-} $, where $ u_b^{-}>0 $ and   $\rho_b^\pm>0$;
  \item the pressure of both the outer and inner layers   is given by the same positive constant $ P_b $;
\item the  flows in the   outer and inner layers are subsonic, i.e.,
 \begin{equation*}
 (u_b^{-})^2<\frac{\gamma P_b}{\rho_b^-}.
\end{equation*}
\end{enumerate}
Then
\begin{equation}\label{2-1}
 \bm{U}_b=
  \begin{cases}
   \bm{U}_b^{+}:=(\rho_b^{+},0,0,0,P_b),  \quad &{\rm{for}}\quad \frac{1}{2}< r<1,\\
   \bm{U}_b^{-}:=(\rho_b^{-},u_{b}^{-},0,0,P_b),  \quad &{\rm{for}}\quad 0\leq r<\frac{1}{2},\\
  \end{cases}
 \end{equation}
 with a contact discontinuity on the surface $ r=\frac{1}{2} $ satisfy  the steady Euler system \eqref{1-1} in the sense of  Definition {\rm{1.1}},
  which will be called the background solutions in this paper.
  This paper is going to establish the structural
stability of these background  solutions  under the suitable axisymmetric perturbations of boundary conditions.
\subsection{The stability problem and the main result}\noindent
\par Let
 $(x, r, \theta)$ be the cylindrical coordinates of $ (x_1, x_2, x_3)\in \mathbb{R}^3 $,  that is
\begin{equation*}
x=x_1,\ r=\sqrt {x_2^2 + x_3^2},\ \theta  = \arctan \frac{{{x_3}}}{{{x_2}}}.
\end{equation*}
Any function $f({\bf x})$ can be represented as $f({\bf x})=f(x,r,\theta)$, and a vector-valued function ${\bf h}({\bf x})$ can be represented as ${\bf h}({\bf x})=h_x(x,r,\theta)\,{\bf e}_x+ h_r(x,r,\theta)\,{\bf e}_r+ h_{\theta}(x,r,\theta)\,{\bf e}_{\theta}$,
where
\begin{equation*}
{\bf e}_x=(1,0,0),\quad {\bf e}_r=(0,\cos\theta, \sin\theta),\quad {\bf e}_{\theta}=(0,-\sin\theta,\cos\theta).
\end{equation*}
We say that a function $f({\bf x})$ is  axisymmetric if its value is independent of $\th$ and that a vector-valued function ${\bf h}= (h_x, h_r, h_{\th})$ is axisymmetric if each of functions $h_x({\bf x}), h_{r}({\bf x})$ and $h_{\th}({\bf x})$ is axisymmetric.
\par Assume that
\begin{equation*}
\begin{aligned}
\rho({\textbf x})&=\rho(x,r),\quad P({\textbf x})=P(x,r),\quad A({\textbf x})=A(x,r),\\
{\bf u}({\textbf x})&= u_x(x,r) {\bf e}_{x}+ u_{r}(x,r) {\bf e}_{r}+ u_{\theta}(x,r) {\bf e}_{\theta},
\end{aligned}
\end{equation*}
then \eqref{1-1} can be  rewritten as
\begin{equation}\label{2-2}
\begin{cases}
\begin{aligned}
&\p_x (\rho u_x) + \p_r (\rho u_r) + \frac{\rho u_r}{r}=0,\\
&\rho (u_x\p_x + u_r \p_r ) u_x + \p_x P = 0,\\
&\rho (u_x\p_x + u_r \p_r ) u_r - \frac{\rho u_{\th}^2}{r} + \p_r P= 0,\\
&\rho (u_x\p_x + u_r \p_r) (ru_\theta)= 0,\\
&\rho (u_x\p_x + u_r \p_r ) A=0.\\
\end{aligned}
\end{cases}
\end{equation}
The    axis and boundary of the cylinder are denoted by $ \Gamma_a $ and $ \Gamma_w $, i.e;
\begin{equation}\label{2-4}
\Gamma_a:=\{(x,r)\in \mathbb{R}^2: 0<x<L,  r=0\},\quad
\Gamma_w:=\{(x,r)\in \mathbb{R}^2: 0<x<L,  r=1\}.
\end{equation}
The exit of the cylinder is denoted by
\begin{equation}\label{2-5}
\Gamma_{L}:=\{(x,r):   x=L,\ 0\leq r<1\}.
\end{equation}
The entrance  of the cylinder  is separated into two parts:
\begin{equation}\label{2-6}
\begin{aligned}
\Gamma_{0}^+:=\{(x,r):   x=0,\ \frac12<r<1\}, \quad \Gamma_{0}^-:=\{(x,r):  x=0,\ 0\leq r<\frac12\}.
\end{aligned}
\end{equation}
At the entrance, we prescribe the boundary data    for  the horizontal mass distribution  $ J=\rho u_x $,  the swirl velocity $  u_{\th} $, the entropy $A $ and the Bernoulli's quantity $ B $:
\begin{equation}\label{2-7}
(J,u_{\th},A,B)(0,r)=
\begin{cases}
  (0,0,A_{b}^+,B_{b}^+),\quad &{\rm{on}}\quad \Gamma_0^+,\\
  (J_{0},\nu_{0},A_{0},B_{0})(r),\quad & {\rm{on}}\quad \Gamma_0^-,\\
  \end{cases}
  \end{equation}
 where
 \begin{equation*}
 A_b^+=\frac{P_b}{(\rho_b^+)^{\gamma}}, \quad B_b^+=\frac{\gamma P_b}{(\gamma-1)\rho_b^+},
 \end{equation*}
 and functions
 $ (J_{0},\nu_{0},A_{0},B_{0})(r)\in
  \left(C^{1,\alpha}([0,\frac12])\right)^4  $ are  close to the background solutions  in some sense that will be clarified later. Moreover, the compatibility conditions hold:
  \begin{equation}\label{2-8}
\nu_{0}(0)=\p_r(\nu_{0},A_{0}, {B}_{0})(0)=0,
 \end{equation}
 since    $(\nu_{0},A_{0},B_{0})(r)$ are $C^{1,\alpha}$ in $\Gamma_0^-$.
  At the exit, the  following boundary condition is satisfied:
  \begin{equation}\label{2-9}
{\bf u}\cdot {\bf e}_r= 0, \quad {\rm{on}}\quad \Gamma_L.
  \end{equation}
\par We expect the flow in the cylinder  will be separated by a  contact discontinuity  $ \Gamma:=\{r=g_{cd}(x),0<x<L\} $ with $ g_{cd}(0)=\frac12 $, which divides the domain $ \mn $ into the subsonic and subsonic regions:
\begin{equation}\label{2-10}
\mn^+:=\mn\cap\{g_{cd}(x)<r<1\}, \quad \mn^-:=\mn\cap\{ 0\leq r<g_{cd}(x)\}.
\end{equation}
Let
\begin{equation*}
{\bm {U}}(x,r)=
  \begin{cases}
   {\bm {U}}^+(x,r):=(\rho^+,u_x^+,u_{r}^+,u_\th^+,P^+)(x,r)\quad {\rm{in}}\quad   \mn^+,\\
   {\bm {U}}^-(x,r):=(\rho^-,u_x^-,u_{r}^-,u_\th^-,P^-)(x,r)  \quad {\rm{in}}\quad  \mn^-.\\
  \end{cases}
 \end{equation*}
 Along the contact discontinuity $ r=g_{cd}(x) $, the following Rankine-Hugoniot conditions hold:
\begin{equation}\label{2-11}
{\bf u} \cdot {\bf n}_{cd}=0, \quad P^+=P^-, \quad {\rm{on}}\quad \Gamma,
\end{equation}
where
\begin{equation*}
{\bf n}_{cd}=\frac{-g_{cd}^\prime(x){\bf e}_{x}+{\bf e}_{r}}{\sqrt{1+|g_{cd}^\prime(x)|^2}}.
\end{equation*}
On the nozzle wall $ \Gamma_w$ , the flow satisfies the slip condition $  \textbf{u}^+\cdot \textbf{n}^+=0 $, where $ \textbf{n}^+ $ is the outer normal vector of the nozzle  wall. Using cylindrical coordinates, the slip boundary condition  can be rewritten as
\begin{equation}\label{2-12}
 u_r^+(x,1)=0, \quad {\rm{on}}\quad \Gamma_w.
\end{equation}
Moreover, since the flow is smooth near the axis $ \Gamma_a $, thus we have the following compatibility conditions:
\begin{equation}\label{2-13}
u_r^-(x,0)=u_\th^-(x,0)=0, \quad \forall x\in[0,L].
\end{equation}
\par In summary, we will investigate the following problem:
 \begin{problem}
  Given  functions $ (J_{0},\nu_{0},A_{0},B_{0})(r)$ at the entrance satisfying \eqref{2-8}, find a unique piecewise smooth axisymmetric subsonic solution $ (\bm{U}^+ ,\bm{U}^-) $ defined on $ \mn^+ $ and $ \mn^- $ respectively,  with the contact discontinuity $ \Gamma $ satisfying the axisymmetric  Euler system \eqref{2-2} in the sense of Definition {\rm{1.1}}  and the Rankine-Hugoniot conditions in \eqref{2-11} and the slip boundary condition in \eqref{2-12} and the compatibility conditions \eqref{2-13}.
  \end{problem}
  \par It is easy to see that $ \bm U_b^+ $ satisfies the following properties:
\begin{itemize}
 \item $\rho_b^{+}>0 $ and $ 0<\sqrt{\frac{\gamma P_b}{\rho_b^+}} $;
 \item  $ A_b^+=\frac{P_b}{(\rho_b^+)^{\gamma}} $ and $ B_b^+=\frac{\gamma P_b}{(\gamma-1)\rho_b^+} $;
      \item $  \bm 0\cdot \bm v=0 $ for any vector $ \bm v\in \mathbb{R}^2$.
  \end{itemize}
  From this observation, we fix ${\bm {U}}^+(x_1,x_2)=  \bm U_b^+$ in $ \mn^+ $ and  solve the following free boundary value problem:
    \begin{problem}
   Under the assumptions of Problem {\rm {2.1}}, find a smooth axisymmetric subsonic solution $\bm{U}^- $  defined on $ \mn^- $  with the contact discontinuity $ \Gamma: r=g_{cd}(x) $ such that the following hold.
    \begin{enumerate}[ \rm (a)]
 \item $ g_{cd}(0)=\frac12 $.
 \item The flow has positive density in the inner cylinder, i.e, $ \rho^->0 $.
  \item Along the contact discontinuity $ r=g_{cd}(x) $, the following Rankine-Hugoniot conditions hold:
      \begin{equation}\label{2-14}
{\bf u}^-\cdot {\bf n}_{cd}=0, \quad P^-=P_b, \quad {\rm{on}}\quad \Gamma.
\end{equation}
\item On the axis $ \Gamma_a $,  the following compatibility conditions hold:
\begin{equation}\label{2-15}
u_r^-(x,0)=u_\th^-(x,0)=0, \quad \forall x\in[0,L].
\end{equation}
\item On the exit,  the following boundary condition holds:
\begin{equation}\label{2-16}
u_r^-(L,r)=0.
\end{equation}
\end{enumerate}
\end{problem}
\begin{figure}
  \centering
  \includegraphics[width=11cm,height=5cm]{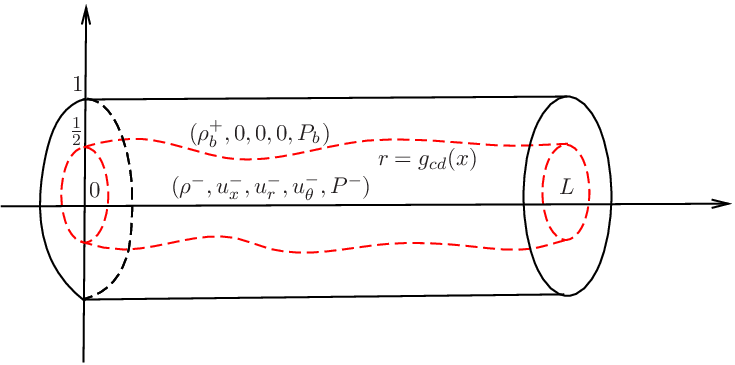}
  \caption{Problem \rm 2.2}
\end{figure}
    \par Before we state our  main result, some weighted H\"{o}lder norms are first introduced: For  a bounded connected open set $\md\in \mathbb{R}^3$, let $ \mb  $ be a closed portion of $\p \md$.  For $ \mathbf{x},\tilde{\mathbf{x}}\in \md $, define
  \begin{equation*}
  \begin{aligned}
 &\delta_{\mathbf{x}}:=\rm{dist}(\mathbf{x},\mb), \quad
  \delta_{\textbf{x},\tilde{\textbf{x}}}:
  =\min\{\delta_{\mathbf{x}},\delta_{\tilde{\mathbf{x}}}\},\\
  \end{aligned}
  \end{equation*}
  Given positive integer $ m $, $ \alpha\in(0,1) $ and $ k\in \mathbb{R} $,  we
  define
  \begin{equation*}
  \begin{aligned}
  {\|u\|}_{m,0;\md}^{(k,\mb)}\ &:=\sum_{0\leq|\beta|\leq m}\sup_{\bm x\in \md} \delta_{\mathbf{x}}^{\max(|\beta|+k,0)}
  |D^{\beta}u(\bm x)|;\\
  [u]_{m,\alpha;\md}^{(k,\mb)}\ &:=\sum_{|\beta|=m}\sup_{\bm x,\tilde{\bm x}\in \md,\bm x\neq \tilde{\bm x}}\delta_{\mathbf{x},\tilde{\mathbf{x}}}^{\max(m+\alpha+k,0)}
  \frac{|D^{\beta}u({\bm x})-D^{\beta}u(\tilde{\bm x})|}{|\bm x-\tilde{\bm x}|^{\alpha}};\\
  \|u\|_{m,\alpha;\md}^{(k,\mb)}\ &:=\|u\|_{m,0;\md}^{k,\mb}
  +[u]_{m,\alpha;\md}^{k,\mb}\\
  \end{aligned}
  \end{equation*}
 with the corresponding function space defined as
   \begin{equation*}
  C_{m,\alpha}^{(k,\mb)}(\md)=\{ u: \|u\|_{m,\alpha;\md}^{(k,\mb)}<\infty\}.
\end{equation*}
For a vector function $ \bm u=(u_1,u_2,\cdots,u_n) $, define
\begin{equation*}
   \|\bm u\|_{m,\alpha;\md}^{(k,\mb)}:=\sum_{i=1}^{n}\| u_i\|_{m,\alpha;\md}^{(k,\mb)}.
\end{equation*}
   \par
The main theorem of this paper can be stated as follows.
\begin{theorem}
Given  functions $ (J_{0},\nu_{0},A_{0},B_{0})(r)$ at the entrance satisfying \eqref{2-8} and $ \alpha \in (\frac12,1) $, we define
\begin{equation}\label{2-17}
   \begin{aligned}
   \sigma(J_{0},\nu_{0},A_{0},B_{0}):&=
   \|(J_{0},\nu_{0},A_{0},B_{0}) -(J_{b}^-,0,A_{b}^-,B_{b}^-\|_{1,\alpha;[0,\frac12]},
   \end{aligned}
   \end{equation}
   where
    \begin{equation*}
 J_b^-=\rho_b^-u_b^-,\quad A_b^-=\frac{P_b}{(\rho_b^-)^{\gamma}}, \quad B_b^-=\frac12(u_b^-)^2+ \frac{\gamma P_b}{(\gamma-1)\rho_b^-}.
 \end{equation*}
 There exist positive constants $\sigma_1 $ and $ \mc^\ast $ depending only on  $ (\bm{U}_b^{-},L,\alpha) $ such that  if
   \begin{equation}\label{2-18}
   \begin{aligned}
  \sigma(J_{0},\nu_{0},A_{0},B_{0})\leq \sigma_1,
   \end{aligned}
   \end{equation}
  $\mathbf{Problem \ 2.2} $ has a unique   smooth axisymmetric subsonic solution  $ \bm{U}^- $   with the contact discontinuity $ \Gamma: r=g_{cd}(x) $ satisfying the following properties:
    \begin{enumerate}[\rm(i)]
\item The  axisymmetric subsonic solution  $ \bm{U}^-\in  C_{1,\alpha}^{(-\alpha,\Gamma)}(\mn^-) $  satisfies the following estimate:
   \begin{equation}\label{2-19}
 \|\bm{U}^- -\bm{U}_b^-\|_{1,\alpha;\mn^-}^{(-\alpha,\Gamma)}\leq \mc^\ast \sigma(J_{0},\nu_{0},A_{0},B_{0}).
\end{equation}
\item The contact discontinuity surface $ g_{cd}(x)\in C^{1,\alpha}([0,L])  $ satisfies
    $g_{cd}(0)=\frac12 $. Furthermore, it holds that
\begin{equation}\label{2-20}
 \|g_{cd}-\frac12 \|_{1,\alpha;[0,L]}\leq \mc^\ast\sigma(J_{0},\nu_{0},A_{0},B_{0}).
 \end{equation}
 \end{enumerate}
\end{theorem}
\begin{remark}
\par  There are several differences between  our result and the  previous work \cite{PB19}.  The first one is that the boundary conditions imposed on the entrance and exit of   the   cylinder. We prescribe the boundary data    for  the horizontal mass distribution,  the swirl velocity, the entropy  and the Bernoulli's quantity at the entrance,  while  \cite{PB19} prescribes the entropy, the swirl velocity and the radial velocity. The second one is that the  decomposition of the axisymmetric Euler system. In \cite{PB19}, the Helmholtz decomposition of the velocity field plays a crucial role. Instead,  we  utilize the deformation-curl decomposition developed in \cite{WX19} for steady Euler system to effectively decouple the hyperbolic and elliptic modes.
The last one is that the approach to locate contact discontinuity. The contact discontinuity in \cite{PB19} is determined by an ordinary differential equation arising from the Rankine-Hugoniot conditions. In this paper, we employ the implicit function
theorem to locate the contact discontinuity.
\end{remark}
\begin{remark}
In $\mathbf{Problem \ 2.2}$,  by fixing the outer-layer flow in $ \mn^+ $ as the background flow $ (\rho_b^{+},0,0,0,P_b) $, we seek a  smooth axisymmetric  subsonic solution $\bm{U}^- $  defined on $ \mn^- $  with the contact discontinuity $ \Gamma: r=g_{cd}(x) $. One can also find a smooth axisymmetric subsonic solution $\bm{U}^+ $  defined on $ \mn^+ $  with the contact discontinuity $ \Gamma: r=g_{cd}(x) $ by fixing the inner-layer flow in $ \mn^- $ as the background flow $ (\rho_b^{-},0,0,0,P_b) $. In fact, this case is even simpler than $\mathbf{Problem \ 2.2}$ since the  singularity near the axis is not needed to be considered. Thus we can introduce the usual Lagrangian transformation and  reduce the axisymmetric Euler system to a second order elliptic equation
for the stream function as  in \cite{WZ23} to obtain the solution $\bm{U}^+ $.
\end{remark}
\section{The reformulation of Problem 2.2}\noindent
\par In this section, we first introduce the modified Lagrange transformation to straighten the contact discontinuity and reformulate the free boundary value    problem 2.2. Then   the deformation-curl decomposition in \cite{WX19,WS19} is employed  to derive an equivalent system.  Finally, we state  the main steps to solve the free boundary problem 3.1.
\subsection{Reformulation by the modified Lagrangian transformation}\noindent
\par For steady Euler flows, the main advantage of the Euler-Lagrange coordinate transformation is to straighten the stream lines. However, in the three-dimensional axisymmetric setting,  there is a singular term $ r $ in the density equation.  We introduce the modified Lagrange transformation to overcome this difficulty and apply this modified Lagrange transformation to straighten the contact discontinuity.
\par  Let $ ( \bm{U}^-(x,r), g_{cd}(x)) $ be a solution to  $ \mathbf{Problem \ 2.2}$.
 Define
 \begin{equation}\label{3-1}
 m^2=\int_{0}^{\frac{1}{2}}sJ_{0}(s)\de s.
  \end{equation}
    For any $x\in( 0,L)$, it follows from the first equation in \eqref{2-2} that
 \begin{equation}\label{3-2}
 \int_{0}^{g_{cd}(x)}s\rho^- u_{x}^-(x,s)\de s=m^2.
 \end{equation}
  \par  Define the modified Lagrangian transformation   as
  \begin{equation}\label{3-4}
y_1=x,\quad  y_2(x,r)= \left(\int_{0}^{r}s\rho^- u_{x}^-(x,s)\de. s\right)^{\frac12}.
  \end{equation}
 Note that if $ (\rho^-,u_x^-, u_r^-,u_{\th}^-) $ is close to the background solutions
$ (\rho_b^-,u_b^-,0,0) $,  there exist two positive constants $ C_1 $ and $ C_2  $, depending on the background solution, such that
\begin{equation*}
C_1r^2\leq  \int_{0}^{r}s\rho^- u_x^-(x,s)\de s\leq C_2r^2.
\end{equation*}
Hence the Jacobian of the modified  Lagrange transformation satisfies
 \begin{equation}\label{3-5}
\frac{\p(y_1,y_2)}{\p(x,r)}=\left|
\begin{matrix} 1& 0 \\ -\frac{r\rho^- u_r^-}{2 y_2} & \frac{r\rho^- u_x^-}{2  y_2} \end{matrix}\right|=\frac{r\rho^- u_x^-}{2   y_2}\geq C_3>0.
\end{equation}
That is invertible.
\par Under this transformation,  the domain $ \mn^- $ becomes
\begin{equation*}
\Omega:=\{(y_1,y_2): 0<y_1<L,\ 0<y_2<m\}.
\end{equation*}
The entrance and exit  of $ \Omega $ are defined as
\begin{equation*}
\begin{aligned}
\Sigma_0:=\{(y_1,y_2):  y_1=0,\ 0<y_2<m\},\quad
\Sigma_L:=\{(y_1,y_2):  y_1=L,\ 0<y_2<m\}.
\end{aligned}
\end{equation*}
 The axis $ \Gamma_a $ is  transformed  into
\begin{equation*}
\begin{aligned}
\Sigma_a:=\{(y_1,y_2): 0<y_1<L,\ y_2=0\}.
\end{aligned}
\end{equation*}
Moreover, on $ \Gamma $, one has
\begin{equation*}
y_2(x,g_{cd}(x))=\left(\int_{0}^{g_{cd}(x)}s\rho^- u_{x}^-(x,s)\de s
\right)^{\frac12}=m.
\end{equation*}
Hence the free boundary $\Gamma $ becomes the following fixed straight line
\begin{equation}\label{3-6}
\Sigma:=\{(y_1,y_2):0<y_1<L,\ y_2=m\}.
\end{equation}
\par In the following,  the superscript ``-" in $\rho^-,u_x^-,u_{r}^-,u_{\th}^-, P^- $ will be ignored to  simplify the notations.  Let
\begin{equation}\label{3-7}
{\bm{U}}^-(y_1,y_2):=( \rho,u_x,u_{r},u_{\th}, P)(y_1,r(y_1,y_2))  \quad {\rm{in}}\quad  \Omega.
 \end{equation}
 Then the  axisymmetric Euler system \eqref{2-2} in the new coordinates  can be rewritten as
 \begin{equation}\label{3-8}
  \begin{cases}
  \begin{aligned}
  &\p_{y_1}\left(\frac{2y_2}{r\rho u_x}\right)-\p_{y_2}\left(\frac{u_r}{u_x}\right)=0,\\
  &\p_{y_1}\left(u_x+\frac{P}{\rho u_x}\right)-\frac{r}{2y_2}\p_{y_2}\left(\frac{Pu_r}{u_x}\right)-\frac{Pu_r}{r\rho u_x^2}=0,\\
  &\p_{y_1}u_r+\frac{r}{2y_2}\p_{y_2}P-\frac{u_{\th}^2}{ru_x}=0,\\
  &\p_{y_1}(ru_\theta)=0,\\
  &\p_{y_1}A=0.\\
  \end{aligned}
  \end{cases}
  \end{equation}
   The background solutions in the  Lagrange coordinates are
\begin{equation}\label{3-9}
 {\bm{U}}_b^{-}:=(\rho_b^{-},u_b^{-},0,0,P_b),  \quad {\rm{in}}\quad \Omega_b,
  \end{equation}
  where
 $\Omega_b:=\{(y_1,y_2): 0<y_1<L,\ 0<y_2<m_b^-\}$
and $ m_b^-=\sqrt{\frac{\rho_b^- u_b^{-}}{8}}$. Without loss of generality, we assume that
\begin{equation*}
 \rho_b^- u_b^{-}=2  \quad {\rm{and}} \quad  m=m_b^-=\frac12.
\end{equation*}
\par Furthermore, under the modified  Lagrangian transformation, $ r $ as a function of $(y_1,y_2) $ becomes nonlinear and nonlocal in the new coordinates. In fact, it follows from the inverse transformation that
\begin{equation*}
\begin{aligned}
  \frac{\p r}{\p y_1}=\frac{u_r}{u_x},\quad  \frac{\p r}{\p y_2}=\frac{2 y_2}{r\rho u_x}, \quad r(y_1,0)=0. \\
 \end{aligned}
  \end{equation*}
  Thus one derives
\begin{equation}\label{3-10}
r(y_1,y_2)=\left(2\int_{0}^{y_2}\frac{2s}{ \rho u_x(y_1,s)}\de s\right)^{\frac{1}{2}}, \ {\rm{in}}\quad \Omega.\\
\end{equation}
In particular, for the background solutions $ (\rho_b^-,u_b^-,0,0) $, one has
\begin{equation}\label{3-11}
r_b(y_2)
=\sqrt{{\frac{2}{\rho_b^- u_b^-}}}y_2=y_2, \ {\rm{in}}\quad  \Omega.
\end{equation}
\par In the new coordinates,  the boundary data \eqref{2-7} at the entrance is given by
 \begin{equation}\label{3-12}
(J,u_{\th},A,B)(0,y_2)=
(\e J_{0},\e \nu_{0},\e A_{0},\e B_{0})(y_2),\quad {\rm{on}}\quad \Sigma_0,
 \end{equation}
 where
 \begin{equation*}
   (\e J_{0},\e \nu_{0},\e A_{0},\e B_{0})(y_2)=( J_{0}, \nu_{0},A_{0}, B_{0})\left(\left(\int_{0}^{y_2}\frac{2s}{ J_{0}(s)}\de s\right)^{\frac{1}{2}}\right).
\end{equation*}
 The Rankine-Hugoniot conditions in \eqref{2-14}  become
\begin{equation}\label{3-13}
\frac{ u_r}{ u_x}(y_1,\frac12)=g_{cd}^{\prime}(y_1),
\end{equation}
and
\begin{equation}\label{3-14}
 P(y_1,\frac12)=P_b.
\end{equation}
\subsection{The deformation-curl decomposition for  axisymmetric Euler system}\noindent
\par  It is well-known that the steady Euler system is  elliptic-hyperbolic coupled in subsonic region, to construct a well-defined iteration scheme, one should decompose the hyperbolic and elliptic modes effectively. Different from the pervious decomposition in \cite{WZ23}, we will employ the deformation-curl decomposition developed in \cite{WX19,WS19} to deal with the elliptic-hyperbolic coupled structure in the axisymmetric  Euler system.
\par First, using the  Bernoulli's function,  the density $ \rho $ can be  represented as
\begin{equation}\label{3-15}
\rho=H(B,A,|\textbf{u}|^2)=
\left(\frac{\gamma-1}{A\gamma}(B-\frac{1}{2}|\textbf{u}|^2)\right)
^{\frac{1}{\gamma-1}}.
\end{equation}
Define the vorticity $ \omega=\curl \textbf{u}=\omega_x \mathbf{e}_x + \omega_r \mathbf{e}_r + \omega_{\theta} \mathbf{e}_{\theta}$, where
 \begin{equation*}
 \omega_x= \frac{1}{r}\p_r(ru_{\th} ),\ \omega_r=-\p_x u_{\th},\ \omega_{\theta}= \p_x u_{r}-\p_r u_{x}.
 \end{equation*}
 From the third equation in \eqref{2-2} and the Bernoulli's law, one derives that
\begin{equation}\label{3-16}
\omega_{\theta}= \frac{u_{\th}\p_ru_{\th}+\frac{u_{\th}^2}{r}+\frac{ (B-\frac{1}{2}|{\bf u}|^2)}{A\gamma}\p_r A-\p_rB}{u_{x}}.
\end{equation}
Substituting \eqref{3-15} into the density equation in \eqref{2-2}, the axisymmetric Euler system \eqref{2-2}  is  equivalent to the following system:
\begin{equation}\label{3-17}
\begin{cases}
\begin{aligned}
&(c^2(H,A)-u_x^2)\p_x u_x+(c^2(H,A)-u_r^2)\p_r u_x-u_x u_r(\p_xu_r+\p_ru_x)+u_r\frac{c^2(H,A)+u_{\th}^2}{r}=0,\\
&u_x(\p_xu_r-\p_ru_x)=u_{\th}\p_ru_{\th}+\frac{u_{\th}^2}{r}+\frac{ (B-\frac{1}{2}|{\bf u}|^2)}{A\gamma}\p_r A-\p_rB,\\
&(u_x\p_x+u_r\p_r)(ru_\theta)=0,\\
&(u_x\p_x+u_r\p_r)A=0,\\
&(u_x\p_x+u_r\p_r)B=0.\\
\end{aligned}
\end{cases}
\end{equation}
\par  Under the  transformation  \eqref{3-4},
  $ A $ and $ B $ satisfy the following transport equations:
\begin{equation}\label{3-18}
\begin{cases}
\begin{aligned}
&\p_{y_1}A=0,\\
&\p_{y_1}B=0.\\
\end{aligned}
\end{cases}
\end{equation}
Thus one has
\begin{equation}\label{3-19}
A(y_1,y_2)=\e A_{0}(y_2),  \quad B(y_1,y_2)=\e B_{0}(y_2).
\end{equation}
Next, it follows from \eqref{3-4} and \eqref{3-19} that $ u_x $, $ u_r $ and $u_\theta$ satisfy the following system:
\begin{equation}\label{3-20}
\begin{cases}
\begin{aligned}
&(c^2(\rho,\e A_{0})-u_x^2)\left(\p_{y_1}u_x-\frac{r\rho u_r}{2y_2}\p_{y_2}u_x\right)
+(c^2(\rho,\e A_{0})-u_r^2)\left(\frac{r\rho u_x}{2y_2}\p_{y_2}u_r\right)\\
&+\frac{c^2(\rho,\e A_{0})}{r}u_r+\frac{u_\theta^2}{r}u_r=
u_xu_r\left(\p_{y_1}u_r-\frac{r\rho u_r}{2y_2}\p_{y_2}u_r+\frac{r\rho u_x}{2y_2}\p_{y_2}u_x\right), \\
&u_x\left(\p_{y_1}u_r-\frac{r\rho u_r}{2y_2}\p_{y_2}u_r-\frac{r\rho u_x}{2y_2}\p_{y_2}u_x\right)\\
&=\frac{r\rho u_x}{2y_2}u_\theta\p_{y_2}u_\theta
+\frac{u_\theta^2}{r}+
\frac{r\rho u_x}{2y_2}\frac{\rho^{\gamma-1}}{\gamma-1}\p_{y_2}\e A_{0}-\frac{r\rho u_x}{2y_2}\p_{y_2}\e B_{0}, \\
&\p_{y_1}(ru_\theta)=0,\\
\end{aligned}
\end{cases}
\end{equation}
with the following boundary conditions:
\begin{equation}\label{3-21}
\begin{cases}
\rho u_x(0,y_2 )=\e J_0(y_2), \ u_\theta(0,y_2)=\e \nu_{0}(y_2),\quad &{\rm{on}}\quad \Sigma_0,\\
u_r(L,y_2)=0, \quad &{\rm{on}}\quad \Sigma_L,\\
u_r(y_1,\frac12)=u_x(y_1,\frac12)g_{cd}^\prime(y_1),\quad &{\rm{on}}\quad \Sigma,\\
u_r(y_1,0)=0, \quad &{\rm{on}}\quad \Sigma_a.\\
\end{cases}
\end{equation}
Furthermore, by \eqref{3-14}, one obtains
\begin{equation}\label{3-22}
 \e A_{0}\left(\frac12\right)\left(\rho(u_x,u_r,u_\theta, \e A_{0},\e B_{0}) \right)^\gamma(y_1,\frac12)=P_b.
\end{equation}
 \par Therefore  $\mathbf{Problem \ 2.2}$ is reformulated as follows.
 \begin{problem}
  Given  functions $ (J_{0}, \nu_{0},A_{0}, B_{0})$ at the entrance satisfying \eqref{2-8}, find a unique  smooth   subsonic solution $ (u_x,u_r,u_\theta;g_{cd}) $ satisfying \eqref{3-20} and \eqref{3-21} and  the Rankine-Hugoniot condition  \eqref{3-22}.
  \end{problem}
  \subsection{Solving the free boundary problem 3.1}\noindent
\par   Note that $\mathbf{Problem \ 3.1}$ is a free boundary problem since the function $ g_{cd} $ is unknown,  this free
boundary problem will be solved by using the implicit function theorem. We follow the steps below to solve $\mathbf{Problem \ 3.1}$:
 \begin{enumerate}[ \bf (a)]
 \item
 Given any function $ g_{cd}(y_1)=\int_{0}^{y_1}w(s)\de s+\frac12 $ belonging to some suitable function classes, we  will solve the nonlinear system \eqref{3-20} with  mixed boundary condition \eqref{3-21} in $ \Omega $. This will be achieved by decomposing the system \eqref{3-20} into two boundary value problems with different inhomogeneous terms  and employing the standard elliptic theory. The detailed analysis will be given in Section 4.
 \item We use  the implicit function theorem to locate the contact discontinuity. More precisely, define the map $  \mq(\bm \varphi_0,w):= \e A_{0}\left(\frac12\right)\left(\rho(u_x,u_r,u_\theta, \e A_{0},\e B_{0}) \right)^\gamma(y_1,\frac12)-P_b $, we need to  compute   the Fr\'{e}chet derivative $ D_w\mq(\bm \varphi_0,w)$ of the functional $ \mq(\bm \varphi_0,w)$ with respect to $ w $ and show that $D_w\mq(\bm \varphi_b,0) $ is an  isomorphism. This step will be achieved in  Section 5.
           \end{enumerate}
\section{The  solution to a fixed boundary value problem in $ \Omega$}\noindent
\par In this section,  given any function $ g_{cd}(y_1)=\int_{0}^{y_1}w(s)\de s+\frac12 \in C^{1,\alpha}([0,L])$ satisfying $ g_{cd}^\prime(L)=0 $,
 we  will solve the nonlinear  system \eqref{3-20} with  mixed boundary condition \eqref{3-21} in $\Omega $.
 \subsection{Linearization}\noindent
 \par To solve nonlinear system  \eqref{3-20} in the  domain $ \Omega $, we first linearize \eqref{3-20} and then solve
the linear system in the  domain $ \Omega $.
  Define
\begin{equation*}
W_1= u_x-u_b^-, \quad W_2=u_r,\quad W_3=u_\theta,\quad (\h J,\h A,\h B)=(\e J_0,\e A_{0},\e B_{0})-(J_b^-, A_{b}^-,B_{b}^-).
\end{equation*}
Denoting the solution space by $\mj(\delta_1)$, which is defined as
\begin{equation}\label{4-1}
\mj(\delta_1)= \{{\bf W}=(W_1,W_2,W_3): \sum_{j=1}^3 \|W_j\|_{1,\alpha;\Omega}^{(-\alpha,\Sigma)}\leq \delta_1, \ W_2(y_1,0)=W_2(L,y_2)=W_3(y_1,0)=0\}.
\end{equation}
Here $ \delta_1 $ is a positive constant to be determined later.
 \par Given $\h{{\bf W}}\in \mj(\delta_1) $, it follows from the third equation in \eqref{3-19} that $W_3$ can be solved as follows
  \begin{equation}\label{4-2}
  W_3(y_1,y_2)=\frac{\h \Lambda(y_2)}{\h r(y_1,y_2)},
  \end{equation}
  where
  \begin{equation*}
  \begin{aligned}
 & \h \Lambda(y_2)=r(0,y_2)\e \nu_{0}(y_2)=\left(2\int_{0}^{y_2}\frac{2s}{ \e J_0 (s)}\de s\right)^{\frac{1}{2}}\e \nu_{0}(y_2),\\
 & \h r(y_1,y_2)=\left(2\int_{0}^{y_2}\frac{2s}{ \hat\rho(\h W_1+u_b^-,\h W_2,\h W_3, \e A_{0},\e B_{0})(\h W_1+u_b^-)(y_1,s)}\de s\right)^{\frac{1}{2}}.
  \end{aligned}
  \end{equation*}
  Then one derives that
  \begin{equation}\label{4-a}
   \|W_3\|_{1,\alpha;\Omega}^{(-\alpha,\Sigma)}\leq C\sigma_{cd},
   \end{equation}
   where $\sigma_{cd}=\sigma(J_{0},\nu_{0},A_{0},B_{0})$  is  defined in \eqref{2-17} and $ C>0 $ depends only on  $ (\bm{U}_b^{-},L,\alpha) $.
  \par In the following, we turn to  concern   $ W_1 $ and  $ W_2 $. It follows from the first and second equations in  \eqref{3-20}
that $ W_1 $ and  $ W_2 $ satisfy the following first order elliptic system:
\begin{equation}\label{4-3}
\begin{cases}
\begin{aligned}
&(c^2(\rho_b^-, A_{b}^-)-(u_b^-)^2)\p_{y_1}W_1
+c^2(\rho_b^-, A_{b}^-)\p_{y_2}W_2\\
&+\frac{c^2(\rho_b^-,A_{b}^-)}{r_b}W_2=
F_1(\h{{\bf W}},\n{\h{\bf W}},\hat A,\hat B), \\
&\p_{y_1}W_2-\frac{r_b\rho_b^- u_b^-}{2y_2}\p_{y_2}W_1
=F_2(\h{{\bf W}},\n{\h{\bf W}},\hat A,\hat B), \\
\end{aligned}
\end{cases}
\end{equation}
where
\begin{equation*}
\begin{aligned}
&F_1(\h{{\bf W}},\n{\h{\bf W}},\h A,\h B)\\
&=(c^2(\h \rho,\e A_0)-(u_b^-+\h{W}_1)^2)\frac{\h r\h\rho \h{W}_2}{2y_2}\p_{y_2}\h{W}_1+\h{W}_2
^2\frac{\h r\h\rho (u_b^-+\h{W}_1)}{2y_2}\p_{y_2}\h{W}_2
-\left(\frac{c^2(\h\rho,\e A_{0})}{\hat r}-\frac{c^2(\rho_b^-, A_b^-)}{r_b}\right)\h{W}_2\\
&\quad-
\left((\gamma-1)\h B-\frac{\gamma+1}{2}\h{W}_1^2
-\frac{\gamma-1}{2}\h{W}_2^2-
\frac{\gamma-1}{2}\h{W}_3^2-(\gamma+1)u_b^-\h{W}_1\right)
\p_{y_1}\h{W}_1\\
&\quad-
\left((\gamma-1)\h B-\frac{\gamma-1}{2}\h{W}_1^2
-\frac{\gamma-1}{2}\h{W}_2^2-
\frac{\gamma-1}{2}\h{W}_3^2-(\gamma-1)u_b^-\h{W}_1\right)
\p_{y_2}\h{W}_2\\
&\quad-\frac{\h W_3^2}{\h r}
\h{W}_2+(\h{W}_1+u_b^-)\h{W}_2\left(\p_{y _1}\h{W}_2-\frac{\h r\h\rho \h{W}_2}{2y_2}\p_{y_2}\h{W}_2+\frac{\h r\h\rho (\h{W}_1+u_b^-)}{2y_2}\p_{y_2}\h{W}_1\right),\\
&F_2(\h{{\bf W}},\n{\h{\bf W}},\h A,\h B)\\
&=\frac{1}{u_b^-+\h{W}_1}\left(\frac{
\h r\h\rho (\h{W}_1+u_b^-)}{2y_2}\h W_3\p_{y_2}
\h W_3
+\frac{\h W_3^2}{\h r}+
\frac{\h r\h\rho (\h{W}_1+u_b^-)}{2y_2}\frac{\h\rho^{\gamma-1}}{\gamma-1}\p_{y_2}\h A-\frac{\h r \rho (\h{W}_1+u_b^-)}{2y_2}\p_{y_2}\h B\right)\\
&\quad+\frac{1}{u_b^-+\h{W}_1}\left(\frac{\h r\h\rho \h{W}_2}{2y_2}\p_{y_2}\h{W}_2\right).\\
\end{aligned}
\end{equation*}
 Recalling that $ \p_r(A_{0}, {B}_{0})(0)=0  $, then one obtains
\begin{equation}\label{4-b}
\p_{y_2}(\h A, \h B)(0)=\p_{y_2}(\e A_0, \e B_0)(0)=0.
\end{equation}
Thus for  $\h{{\bf W}}\in \mj(\delta_1) $, it is easy to check that $ F_2(y_1,0)=0 $.
Furthermore, there exist positive constants $ \kappa_1 $ and $ \kappa_2 $ depending only on the background solutions $ \bm{U}_b^{-} $ such that
\begin{equation*}
\kappa_1y_2\leq \h r\leq \kappa_2y_2,
\end{equation*}
which implies that
\begin{equation}\label{4-c}
\kappa_1\leq \frac{\h r}{y_2}\leq \kappa_2.
\end{equation}
Hence one can derive
\begin{equation}\label{4-d}
 \begin{aligned}
 \sum_{j=1}^2 \|F_j\|_{0,\alpha;\Omega}^{(1-\alpha,\Sigma)}
    \leq C\left(\bigg(\sum_{j=1}^3 \|\h {W_j}\|_{1,\alpha;\Omega}^{(-\alpha,\Sigma)}\bigg)^2
    + \sigma_{cd}\right)
   \leq C\left(\delta_1^2
+\sigma_{cd}\right),
    \end{aligned}
    \end{equation}
  where $ C>0 $ depends only on  $ (\bm{U}_b^{-},L,\alpha) $.
\par Next, we derive the boundary conditions for  ${\bf W}$. It follows from \eqref{3-20} that
 \begin{equation}\label{4-4}
\begin{cases}
 W_1(0,y_2)=F_3(\h{{\bf W}},\h J,\h A,\h B)(0,y_2), \quad &{\rm{on}}\quad \Sigma_0,\\
W_2(L,y_2)=0, \quad &{\rm{on}}\quad \Sigma_L,\\
W_2(y_1,\frac12)=F_4(\h{{\bf W}},g_{cd})(y_1,\frac12),\quad &{\rm{on}}\quad \Sigma,\\
W_2(y_1,0)=0, \quad &{\rm{on}}\quad \Sigma_a,\\
\end{cases}
\end{equation}
where
\begin{equation*}
\begin{aligned}
&F_3(\h{{\bf W}},\h J,\h A,\h B)\\
&=\frac{\h J} {\rho_b^-(1-(M_b^-)^2)}- \frac{\h{W}_1}{\rho_b^- (1-(M_b^-)^2)}\bigg(H(\h B+B_b^-,\h A+A_b^-,(u_b^-+\h{W}_1)^2+\h{W}_2^2+
\h{W}_3^2)\\
&\qquad\qquad\qquad\qquad\qquad\qquad\qquad\qquad\quad- H(B_b^-,A_b^-,(u_b^-)^2)\bigg) \\
&\quad - \frac{u_b^-}{\rho_b^- (1-(M_b^-)^2)}\left(H(\h B+B_b^-,\h A+A_b^-,(u_b^-+\h{W}_1)^2+\h{W}_2^2+\h{W}_3^2)
\right.\\
&\qquad\qquad\qquad\qquad\qquad\left.- H(B_b^-,A_b^-,(u_b^-)^2)+ \frac{J_b^-}{c^2(\rho_b^-,A_b^-)} \h{W}_1\right),\\
& (M_b^-)^2=\frac{(u_b^-)^2}{c^2(\rho_b^-,A_b^-)}, \quad F_4(\h{{\bf W}},g_{cd})=
(u_b^-+\h{W}_1)g_{cd}^\prime.
\end{aligned}
\end{equation*}
 Then a direct computation yields
 \begin{equation}\label{4-5}
 \begin{aligned}
 &\|F_3\|_{1,\alpha;[0,\frac12)}^{(-\alpha,\{\frac12\})}+ \|F_4\|_{0,\alpha;[0,L]}\\
    &\leq C\left(\bigg(\sum_{j=1}^3 \|\h {W_j}\|_{1,\alpha;\Omega}^{(-\alpha,\Sigma)}\bigg)^2
    +\|\h {W_1}\|_{1,\alpha;\Omega}^{(-\alpha,\Sigma)}\|g_{cd}^\prime\|_{0,\alpha;[0,L]}
    +\|g_{cd}^\prime\|_{0,\alpha;[0,L]}+ \sigma_{cd}\right)\\
   &\leq C\left(\delta_1^2+\delta_1\|w\|_{0,\alpha;[0,L]}
   +\|w\|_{0,\alpha;[0,L]}
+\sigma_{cd}\right),
    \end{aligned}
    \end{equation}
     where $ C>0 $ depends only on  $ (\bm{U}_b^{-},L,\alpha) $.
     \subsection{Solving the linear boundary value problem  }\noindent
  \par In this subsection, we consider the following linear boundary value problem:
  \begin{equation}\label{4-6}
\begin{cases}
\begin{aligned}
&(1-(M_b^-)^2)\p_{y_1}W_1
+\p_{y_2}W_2
+\frac{1}{y_2}W_2=
\mf_1(\h{{\bf W}},\n{\h{\bf W}},\h A,\h B), \\
&\p_{y_1}W_2-\p_{y_2}W_1
=\mf_2(\h{{\bf W}},\n{\h{\bf W}},\h A,\h B), \\
 &W_1(0,y_2)=\mf_3(\h{{\bf W}},\h J,\h\Lambda,\h A,\h B)(0,y_2), \quad &{\rm{on}}\quad \Sigma_0,\\
&W_2(L,y_2)=0, \quad &{\rm{on}}\quad \Sigma_L,\\
&W_2(y_1,\frac12)=\mf_4(\h{{\bf W}},g_{cd})(y_1,\frac12),\quad &{\rm{on}}\quad \Sigma,\\
&W_2(y_1,0)=0, \quad &{\rm{on}}\quad \Sigma_a,\\
\end{aligned}
\end{cases}
\end{equation}
where
\begin{equation*}
\mf_1=\frac{1}{c^2(\rho_b^-,A_b^-)}F_1, \quad \mf_i=F_i, i=2,3,4.
\end{equation*}
\par For the   problem \eqref{4-6}, we have the following conclusion:
  \begin{lemma}
  Let $ \alpha\in (\frac12,1) $. For given $ \mf_j \in  C_{0,\alpha}^{(1-\alpha,\Sigma)}(\Omega) $, $ j=1,2 $, $ \mf_2(y_1,0)=0 $, $ \mf_3 \in  C_{1,\alpha}^{(-\alpha,\{\frac12\})}([0,\frac12)) $, $ \mf_4 \in  C^{0,\alpha}([0,L]) $, the boundary value problem \eqref{4-6} has a unique solution $ (W_1,W_2) \in \left( C_{1,\alpha}^{(-\alpha,\Sigma)}(\Omega)\right)^2 $ satisfying
  \begin{equation}\label{4-7}
  \sum_{j=1}^2 \|W_j\|_{1,\alpha;\Omega}^{(-\alpha,\Sigma)}\leq
  C\left(\sum_{j=1}^2 \|\mf_j\|_{0,\alpha;\Omega}^{(1-\alpha,\Sigma)}+
    \|\mf_3\|_{1,\alpha;[0,\frac12)}^{(-\alpha,\{\frac12\})}+ \|\mf_4\|_{0,\alpha;[0,L]}\right),
    \end{equation}
 where $ C>0 $ depends only on  $ (\bm{U}_b^{-},L,\alpha) $.
 \end{lemma}
 \begin{proof}
  We divide the proof into four steps.
  \par  { \bf Step 1}:  In this step, in order to solve \eqref{4-6},  we first introduce a transformation to  reduce it  to a typical form.
  \par Let
  \begin{equation}\label{4-8}
\begin{array}{ll}
\left\{
\begin{aligned}
&z_1=  \sqrt{\frac{1}{1-(M_b^-)^2}}y_1 ,\\
& z_2=y_2,\\
\end{aligned}
\right.  \quad \text{and}\quad
\left\{
\begin{aligned}
&V_1=\sqrt{{1-(M_b^-)^2}}W_1,\\
& V_2=W_2.\\
\end{aligned}
\right.
\end{array}
\end{equation}
The domain $\Omega $ becomes
 \begin{equation*}
 \Omega^\ast=\{(z_1,z_2): 0<z_1<L^\ast, 0<y_2<\frac12\},
 \end{equation*}
 where $ L^\ast=\sqrt{\frac{1}{1-(M_b^-)^2}}L $, and its boundaries are transformed into
 \begin{equation*}
\begin{aligned}
&\Sigma_0^\ast:=\{(z_1,z_2):  z_1=0,\ 0\leq z_2< \frac12\},\quad
\Sigma_L^\ast:=\{(z_1,z_2):  z_1=L^\ast,\ 0\leq z_2 <\frac12\},\\
&\Sigma_a^\ast:=\{(z_1,z_2): 0<z_1<L^\ast,\ z_2=0\},\quad
\Sigma^\ast:=\{(z_1,z_2): 0<z_1<L^\ast,\ z_2=\frac12\}.\\
\end{aligned}
\end{equation*}
Then the system \eqref{4-6} is reformulated as
 \begin{equation}\label{4-9}
\begin{cases}
\begin{aligned}
&\p_{z_1}(z_2V_1)
+\p_{z_2}(z_2V_2)
=z_2\mf_1, \\
&\p_{z_1}V_2-\p_{z_2}V_1
=\sqrt{{1-(M_b^-)^2}}\mf_2:=\e\mf_2, \\
 &V_1(0,z_2)=\sqrt{{1-(M_b^-)^2}}\mf_3:=\e\mf_3, \quad &{\rm{on}}\quad \Sigma_0^\ast,\\
&V_2(L^\ast,z_2)=0, \quad &{\rm{on}}\quad \Sigma_L^\ast,\\
&V_2(z_1,\frac12)=\mf_4,\quad &{\rm{on}}\quad \Sigma^\ast,\\
&V_2(z_1,0)=0, \quad &{\rm{on}}\quad \Sigma_a^\ast.\\
\end{aligned}
\end{cases}
\end{equation}
\par Next, we decompose the problem \eqref{4-9} into two boundary value problems with different inhomogeneous terms as follows.
Let $ (V_1,V_2)^T=(H_1,H_2)^T+(K_1,K_2)^T $, where $ (H_1,H_2)^T $ is the solution to the problem
 \begin{equation}\label{4-10}
\begin{cases}
\begin{aligned}
&\p_{z_1}(z_2H_1)
+\p_{z_2}(z_2H_2)
=0, \\
&\p_{z_1}H_2-\p_{z_2}H_1
=\e\mf_2, \\
 &H_1(0,z_2)=0, \quad &{\rm{on}}\quad \Sigma_0^\ast,\\
&H_2(L^\ast,z_2)=0, \quad &{\rm{on}}\quad \Sigma_L^\ast,\\
&H_2(z_1,\frac12)=0,\quad &{\rm{on}}\quad \Sigma^\ast,\\
&H_2(z_1,0)=0, \quad &{\rm{on}}\quad \Sigma_a^\ast,\\
\end{aligned}
\end{cases}
\end{equation}
and $ (K_1,K_2)^T $ satisfies the following problem
\begin{equation}\label{4-11}
\begin{cases}
\begin{aligned}
&\p_{z_1}(z_2K_1)
+\p_{z_2}(z_2K_2)
=z_2\mf_1, \\
&\p_{z_1}K_2-\p_{z_2}K_1
=0, \\
 &K_1(0,z_2)=\e\mf_3, \quad &{\rm{on}}\quad \Sigma_0^\ast,\\
&K_2(L^\ast,z_2)=0, \quad &{\rm{on}}\quad \Sigma_L^\ast,\\
&K_2(z_1,\frac12)=\mf_4,\quad &{\rm{on}}\quad \Sigma^\ast,\\
&K_2(z_1,0)=0, \quad &{\rm{on}}\quad \Sigma_a^\ast.\\
\end{aligned}
\end{cases}
\end{equation}
 \par  { \bf Step 2}:  In this step, we are going to solve \eqref{4-10}. The first equation in \eqref{4-10} implies that there exists a potential function $ \phi_1 $ such that
 \begin{equation}\label{4-12}
 (\p_{z_1}\phi_1,\p_{z_2}\phi_1)=(z_2H_2,-z_2H_1).
 \end{equation}
 Let $ \Phi_1=\frac{\phi_1}{z_2} $. Then \eqref{4-12} yields that
 \begin{equation}\label{4-13}
 (H_1,H_2)=(-(\p_{z_2}\Phi_1+\frac{\Phi_1}{z_2}),\p_{z_1}\Phi_1).
 \end{equation}
 Without loss of generality, we assume that $ \Phi_1(0,0)=0 $. Thus  \eqref{4-10} can be rewritten as the following equation for $ \Phi_1$:
 \begin{equation}\label{4-14}
\begin{cases}
\begin{aligned}
&\p_{z_1}^2\Phi_1+\p_{z_2}^2\Phi_1+\frac{\p_{z_2}\Phi_1}{z_2}
-\frac{\Phi_1}{z_2^2}
=\e\mf_2, \\
 &\p_{z_1}\Phi_1=0, \quad {\rm{on}}\quad \Sigma_L^\ast,\\
&\Phi_1=0,\qquad \ {\rm{on}}\quad \Sigma_a^\ast\cup\Sigma^\ast\cup \Sigma_0^\ast.\\
\end{aligned}
\end{cases}
\end{equation}
\par Obviously, the coefficients of equation \eqref{4-14} tends to infinity as $ z_2 $ goes to $ 0 $. By  applying the idea of Proposition 3.3 in \cite{BW18}, we rewrite \eqref{4-14} as a boundary value problem in $\mathbb{R}^5$ so that the singular term in \eqref{4-14} can be removed from the equation for $ \Phi_1 $. Set
\begin{equation}\label{4-15}
\Phi^\ast=\frac{\Phi_1}{z_2},\quad  \mf_2^\ast=\frac{\e\mf_2}{z_2}.
\end{equation}
 We  regard $\Phi^\ast $ and $ \mf_2^\ast $  as functions defined in
 \begin{equation}\label{4-16}
 \md^\ast:=\{(z_1,\mathbf z_2):z_1\in (0, L^\ast), \mathbf z_2\in \mathbb{R}^4: |\mathbf z_2|\leq\frac12\}\subset \mathbb{R}^5,
 \end{equation}
 where $ \mathbf  z_2=( z_{21},z_{22},z_{23},z_{24}) $ and $ \sum_{j=1}^{4}z_{2j}^2=|\mathbf z_2|^2 $. Define
 \begin{equation*}
 \bm \mg^\ast(z_1,\mathbf  z_2)=(0, \mg^\ast(z_1,\mathbf z_2)z_{21},\mg^\ast(z_1,\mathbf z_2)z_{22},\mg^\ast( z_1, \mathbf z_2) z_{23},\mg^\ast(z_1,\mathbf z_2)z_{24}), \ \forall \mathbf{z}=( z_1,\mathbf z_2)\in \md^\ast,
  \end{equation*}
  with
   \begin{equation*}
  \mg^\ast(z_1,\mathbf z_2)=\frac{1}{|\mathbf z_2|^4}\int_{0}^{\mathbf z_2}s^3   \mf_2^\ast( z_1,s)\de s.
  \end{equation*}
    Then it follows from \eqref{4-14} that
   \begin{equation}\label{4-17}
\begin{cases}
\begin{aligned}
 & \triangle_{\mathbf z}\Phi^\ast
=\text{div}_{\mathbf{z}}\bm \mg^\ast, \qquad {\rm{in }}\quad \md^\ast,\\
 &\p_{z_1}\Phi^\ast(L^\ast,\mathbf{z}_2)=0, \quad {\rm{on}}\quad  B_{L^\ast}:=\{L^\ast\}\times\{\mathbf{z}_2\in  \mathbb{R}^4: |\mathbf{z}_2|\leq\frac12\},\\
&\Phi^\ast(0,\mathbf{z}_2)=0,\qquad\ \  {\rm{on}}\quad  B_0:=\{0\}\times\{\mathbf{z}_2\in  \mathbb{R}^4: |\mathbf{z}_2|\leq \frac12\},\\
&\Phi^\ast(z_1,\mathbf{z}_2)=0,\qquad \ \  {\rm{on}}\quad  B_w:=[0,L^\ast]\times \{\mathbf{z}_2\in  \mathbb{R}^4: |\mathbf{z}_2|=\frac12\}.\\
\end{aligned}
\end{cases}
\end{equation}
\par The standard elliptic theory in \cite{GT98} yields that \eqref{4-17} has a unique weak solution $ \Phi^\ast\in H^1(\md^\ast) $ satisfying
 \begin{equation}\label{4-18}
 \ma[\Phi^\ast,\xi]=(\bm \mg^\ast,\xi) \quad {\rm{for}}\  {\rm{all}} \ \xi\in\{\xi\in H^1(\md^\ast): \xi=0 \quad {\rm{on}}\quad  B_0\cup B_w\},
 \end{equation}
 where
  \begin{equation*}
  \begin{aligned}
  \ma[\Phi^\ast,\xi]=\int_{\md^\ast} \n \Phi^\ast\n \xi\de \mathbf{z},\quad
  (\bm \mg^\ast,\xi)=\int_{\md^\ast} \bm \mg^\ast\n \xi\de \mathbf{z}.
  \end{aligned}
 \end{equation*}
 Furthermore, $  \Phi^\ast $ satisfies
  \begin{equation}\label{4-19}
  \|\Phi^\ast\|_{H^1(\md^\ast)}\leq C\|\bm\mg^\ast\|_{L^2(\md^\ast)}.
  \end{equation}
\par Next, we prove
\begin{equation}\label{4-20}
\|\bm\mg^\ast\|_{L^2(\md^\ast)}\leq C\|\bm\mg^\ast\|_{0,\alpha;\md^\ast}^{(1-\alpha,\p \md^\ast)}.
 \end{equation}
 Note that $ \e\mf_2(z_1,0)=0 $.  Then it holds that
 \begin{equation}\label{4-21}
 \mf_2^\ast(z_1,z_2)=\frac{\e\mf_2(z_1,z_2)-\e\mf_2(z_1,0)}{z_2}
 =\frac{\e\mf_2(z_1,z_2)-\e\mf_2(z_1,0)}{z_2^\alpha}z_2^{\alpha-1}.
 \end{equation}
 Since $ \e\mf_2\in C_{0,\alpha}^{(1-\alpha, \Sigma^\ast)}(\Omega^\ast) $, it is easy to check that
 \begin{equation}\label{4-22}
 \|\bm\mg^\ast\|_{0,\alpha;\md^\ast}^{(1-\alpha,\p \md^\ast)}\leq C
  \|\e\mf_2\|_{0,\alpha;\Omega^\ast}^{(1-\alpha,\Sigma^\ast)}.
 \end{equation}
 By the weighted H$\ddot{\rm{o}}$lder norm, one derives
 \begin{equation}\label{4-23}
 |\bm\mg^\ast(\mathbf{z})|\leq \delta_{\mathbf{z}}^{\alpha-1}\|\bm\mg^\ast\|
 _{0,\alpha;\e\md^\ast}^{({1-\alpha;\p\md^\ast})},
 \end{equation}
 with $ \delta_{\mathbf{z}}={\rm{dist}}(\mathbf{z},\p\md^\ast)$. Thus for $ \alpha \in(\frac{1}{2},1) $, one has
  \begin{equation*}
  \int_{ \md^\ast} |\bm\mg^\ast|^2\de \mathbf{z}\leq C\left(\|\bm\mg^\ast\|_{0,\alpha;\md^\ast}^{({1-\alpha;\p\md^\ast})}
  \right)^2.
   \end{equation*}
  \par  Next, we improve the regularity of $ \Phi^\ast $.  For $ \mathbf{z}_0  \in {\md^\ast} $ and $ \eta\in \mathbb{R } $ with $ 0<\eta<\frac{1}{10} $, set
  \begin{equation*}
  \begin{aligned}
  &B_{\eta}(\mathbf{z}_0):=\{\mathbf{z}\in \mathbb{R}^5:
  |\mathbf{z}_0-\mathbf{z}|<\eta\}, \quad D_{\eta}(\mathbf{z}_0):= B_{\eta}(\mathbf{z}_0)\cap \md^\ast,\\
  & \Phi^\ast_{\mathbf{z}_0,\eta}:=\frac{1}{|D_{\eta}(\mathbf{z}_0)|}
  \int_{D_{\eta}(\mathbf{z}_0)} \Phi^\ast \de \mathbf{z}.
  \end{aligned}
 \end{equation*}
 Note that there exists  a constant $ \lambda_0\in(0,1/10) $ such that
 \begin{equation*}
 \lambda_0\leq \frac{|D_{\eta}(\mathbf{z}_0)|}{|B_{\eta}(\mathbf{z}_0)|}
 \leq \frac{1}{\lambda_0}.
 \end{equation*}
 Hence we follow the proof in Theorem 3.8 of \cite{HL11} to get
 \begin{equation}\label{4-24}
  \int_{D_{\eta}(\mathbf{z})}| \Phi^\ast- \Phi^\ast_{\mathbf{z},\eta}|^2\de \mathbf{z}\leq C\left(\|\bm\mg^\ast\|_{0,\alpha;\md^\ast}^{({1-\alpha;\p\md^\ast})}
  \right)^2
  \eta^{5+2\alpha}
  \end{equation}
  for any $ \mathbf{z}  \in \overline{\md^\ast} $. Once \eqref{4-24} is obtained, it follows from Theorem 3.1 in  \cite{HL11} that
   \begin{equation}\label{4-25}
  \|\Phi^\ast\|_{0,\alpha;\md^\ast}\leq C\|\bm\mg^\ast\|_{0,\alpha;\md^\ast}^{({1-\alpha;\p\md^\ast})}.
  \end{equation}
  \par We proof \eqref{4-24} only for the case $ \mathbf{z}  \in B_0\cap B_w$, since the other cases can be treated similarly. Fix $ \mathbf{z}_0\in B_0\cap B_w $ and $ \chi\in \mathbb{R} $ with $ 0<\chi<\frac{1}{10} $. Let $ \Phi^\ast_h $ be a weak solution of the following problem:
  \begin{equation}\label{4-26}
\begin{cases}
\begin{aligned}
 & \triangle_{\mathbf z}\Phi_h^\ast
=0, \qquad {\rm{in }}\quad D_{\chi}(\mathbf{z}_0),\\
 &\Phi_h^\ast=\Phi^\ast,  \qquad\ \ {\rm{in }}\quad \p D_{\chi}(\mathbf{z}_0)\cap \md^\ast.\\
\end{aligned}
\end{cases}
\end{equation}
Then $ \e\Phi_h^\ast=\Phi^\ast-\Phi_h^\ast $ satisfies
\begin{equation*}
\int_{D_{\chi}(\mathbf{z}_0)} \n \e\Phi_h^\ast\n \xi\de \mathbf{z}=\int_{D_{\chi}(\mathbf{z}_0)} \bm \mg^\ast\n \xi\de \mathbf{z},
\end{equation*}
for any $\xi\in\{\xi\in H^1(D_{\chi}(\mathbf{z}_0): \xi=0 \ {\rm{on}}\  \p D_{\chi}(\mathbf{z}_0) \cap(B_0\cup B_w)\} $.  Taking the test function $ \xi=\e\Phi_h^\ast $ and using the H\"{o}lder inequality to yield that
\begin{equation}\label{4-27}
\int_{D_{\chi}(\mathbf{z}_0)} |\n \e\Phi_h^\ast|^2\de \mathbf{z}\leq\int_{D_{\chi}(\mathbf{z}_0)} |\bm \mg^\ast|^2\de \mathbf{z}.
\end{equation}
Due to $ \bm \mg^\ast\in C_{0,\alpha}^{({1-\alpha;\p\md^\ast})}(\md^\ast) $, one gets
\begin{equation}\label{4-28}
\int_{D_{\chi}(\mathbf{z}_0)} |\n \e\Phi_h^\ast|^2\de \mathbf{z}
\leq\left(\|\bm\mg^\ast\|_{0,\alpha;\md^\ast}^{({1-\alpha;\p\md^\ast})}
  \right)^2\int_{D_{\chi}(\mathbf{z}_0)}\delta_{\mathbf{z}}^{2(\alpha-1)}
  \de  \mathbf{z}\leq\chi^{3+2\alpha}\left(\|\bm\mg^\ast\|_{0,\alpha;\md^\ast}^{({1-\alpha;\p\md^\ast})}
  \right)^2.
  \end{equation}
   By Corollary 3.11 in \cite{HL11}, for $0<\eta<\chi $, one has
  \begin{equation}\label{4-29}
  \begin{aligned}
    \int_{D_{\eta}(\mathbf{z}_0)}|\n \Phi^\ast|^2\de \mathbf{z}&\leq
    C\left(\frac{\eta}{\chi}\right)^5\int_{D_{\chi}(\mathbf{z}_0)}|\n \Phi^\ast|^2\de \mathbf{z}
    +C\left(\|\bm\mg^\ast\|_{0,\alpha;\md^\ast}^{({1-\alpha;\p\md^\ast})}
  \right)^2
  \chi^{3+2\alpha}.
  \end{aligned}
  \end{equation}
  Then it follows from Lemma 3.4 in \cite{HL11} that one obtains
   \begin{equation}\label{4-30}
  \begin{aligned}
    \int_{D_{\eta}(\mathbf{z}_0)}|\n \Phi^\ast|^2\de \mathbf{z}&\leq
    C\left(\frac{1}{\chi^{3+2\alpha}}\int_{\md^\ast}|\n \Phi^\ast|^2\de \mathbf{z}
    +\left(\|\bm\mg^\ast\|_{0,\alpha;\md^\ast}^{({1-\alpha;\p\md^\ast})}
  \right)^2\right)
  \eta^{3+2\alpha}.
  \end{aligned}
  \end{equation}
  By applying Poincar\'{e} inequality, one can derive
  \begin{equation}\label{4-31}
  \int_{D_{\eta}(\mathbf{z}_0)}| \Phi^\ast- \Phi^\ast_{\mathbf{z},\eta}|^2\de \mathbf{z}\leq
    C\left(\frac{1}{\chi^{3+2\alpha}}\int_{\md^\ast}|\n \Phi^\ast|^2\de \mathbf{z}
    +\left(\|\bm\mg^\ast\|_{0,\alpha;\md^\ast}^{({1-\alpha;\p\md^\ast})}
  \right)^2\right)
  \eta^{5+2\alpha}.
\end{equation}
Then it follows from \eqref{4-19} and \eqref{4-20} that
 \begin{equation*}
 \int_{D_{\eta}(\mathbf{z}_0)}| \Phi^\ast- \Phi^\ast_{\mathbf{z},\eta}|^2\de \mathbf{z}
    \leq C\left(\|\bm\mg^\ast\|_{0,\alpha;\md^\ast}^{({1-\alpha;\p\md^\ast})}
  \right)^2
  \eta^{5+2\alpha}.
  \end{equation*}
  Hence the proof of \eqref{4-25} is completed.
  \par Therefore, by  the scaling argument and the Schauder estimate in \cite{GT98}, one has
  \begin{equation}\label{4-32}
 \|\Phi^\ast\|_{1,\alpha;\md^\ast}^{(-\alpha,\p \md^\ast)} \leq C\|\bm\mg^\ast\|_{0,\alpha;\md^\ast}^{(1-\alpha,\p \md^\ast)}\leq C
  \|\e\mf_2\|_{0,\alpha;\Omega^\ast}^{(1-\alpha,\Sigma^\ast)}.
 \end{equation}
Furthermore, the rotational  invariance of the boundary value problem \eqref{4-17} and the  uniqueness of the solution $ \Phi^\ast $ and  the estimate \eqref{4-32} imply that
 \begin{equation*}
 \|\Phi^\ast\|_{1,\alpha;\Omega^\ast}^{(-\alpha,\Sigma^\ast)} \leq C
  \|\e\mf_2\|_{0,\alpha;\Omega^\ast}^{(1-\alpha,\Sigma^\ast)}.
  \end{equation*}
 Using the first equation in \eqref{4-14}, one can verify
 \begin{equation}\label{4-33}
 \p_{z_1}^2\Phi_1+\p_{z_2}^2\Phi_1
=\e\mf_2-\frac{\p_{z_2}\Phi_1}{z_2}+\frac{\Phi_1}{z_2^2}=
\e\mf_2-\p_{z_2}\Phi^\ast\in H_{0,\alpha}^{(1-\alpha,\Sigma^\ast)}(\Omega^\ast).
 \end{equation}
 By the Schauder estimate in Theorem 4.6 of \cite{LG13}, we obtain \begin{equation}\label{4-34}
 \|\Phi_1\|_{2,\alpha;\Omega^\ast}^{(-1-\alpha,\Sigma^\ast)} \leq C
  \|\e\mf_2\|_{0,\alpha;\Omega^\ast}^{(1-\alpha,\Sigma^\ast)}.
  \end{equation}
  Finally, by the definition of $ \Phi_1 $, one has
  \begin{equation}\label{4-35}
  \sum_{j=1}^2 \|H_j\|_{1,\alpha;\Omega^\ast}^{(-\alpha,\Sigma^\ast)}
  \leq C
  \|\e\mf_2\|_{0,\alpha;\Omega^\ast}^{(1-\alpha,\Sigma^\ast)}.
  \end{equation}
  \par  { \bf Step 3}:  In this step, we are going to solve \eqref{4-11}. The second equation in \eqref{4-11} implies that there exists a potential function $ \phi_2 $ such that
 \begin{equation}\label{4-36}
 (\p_{z_1}\phi_2,\p_{z_2}\phi_2)=(K_1,K_2), \quad \phi_2(L^\ast,0)=0.
 \end{equation}
 Then  \eqref{4-11} can be rewritten as the following equation for $ \phi_2$:
 \begin{equation}\label{4-37}
\begin{cases}
\begin{aligned}
&\p_{z_1}(z_2\p_{z_1}\phi_2)
+\p_{z_2}(z_2\p_{z_2}\phi_2)
=z_2\mf_1, \\
 &\p_{z_1}\phi_2(0,z_2)=\e\mf_3, \quad &{\rm{on}}\quad \Sigma_0^\ast,\\
&\phi_2(L^\ast,z_2)=0, \quad &{\rm{on}}\quad \Sigma_L^\ast,\\
&\p_{z_2}\phi_2(z_1,\frac12)=\mf_4,\quad &{\rm{on}}\quad \Sigma^\ast,\\
&\p_{z_2}\phi_2(z_1,0)=0, \quad &{\rm{on}}\quad \Sigma_a^\ast.\\
\end{aligned}
\end{cases}
\end{equation}
\par In order to deal with the singularity near $ z_2=0 $, we rewrite the problem \eqref{4-37} in the three dimensional
setting. Define
\begin{equation*}
\zeta_1=z_1,\ \zeta_2=z_2\cos\tau,\ \zeta_3=z_2\sin\tau,\ \tau\in[0,2\pi],
\end{equation*}
and
\begin{equation*}
\begin{aligned}
&E_1=\{(\zeta_1,\zeta_2,\zeta_3):0<\zeta_1<L^\ast, \zeta_2^2+\zeta_3^2\leq \frac12\}, \quad E_2= \{(\zeta_2,\zeta_3): \zeta_2^2+\zeta_3^2\leq \frac12\},\\
&\Gamma_{w,\zeta}=[0,L^\ast]\times\{(\zeta_2,\zeta_3):\zeta_2^2+\zeta_3^2= \frac12\},\\
&\Gamma_{0,\zeta}=\{0\}\times\{(\zeta_2,\zeta_3):\zeta_2^2+\zeta_3^2\leq \frac12\},\quad
\Gamma_{L^\ast,\zeta}=\{L^\ast\}\times\{(\zeta_2,\zeta_3):
\zeta_2^2+\zeta_3^2\leq \frac12\},\\
&\Psi(\bm \zeta)= \phi_2(\zeta_1,\sqrt{\zeta_2^2+\zeta_3^2})=
\Psi(\zeta_1,|\zeta^\prime|).
\end{aligned}
\end{equation*}
Then $ \Psi $ solves the following problem
 \begin{equation}\label{4-38}
\begin{cases}
\begin{aligned}
&\Delta \Psi
=\mf_1(\zeta_1,|\zeta'|), \\
 &\p_{\zeta_1}\Psi(0,|\zeta'|)=\e\mf_3(|\zeta'|), \quad &{\rm{on}}\quad \Gamma_{0,\zeta},\\
&\Psi(L^\ast,|\zeta'|)=0, \quad &{\rm{on}}\quad \Gamma_{L^\ast,\zeta},\\
&(\zeta_2\p_{\zeta_2}+ \zeta_3\p_{\zeta_3})\Psi(\zeta_1,|\zeta'|)=\frac12\mf_4(\zeta_1),\quad &{\rm{on}}\quad \Gamma_{w,\zeta}.\\
\end{aligned}
\end{cases}
\end{equation}
First, the weak solution to \eqref{4-38} can be obtained as follows.  $ \Psi\in  H^1(E_1) $ is said to be a weak solution to \eqref{4-38} if the following holds
 \begin{equation}\label{4-39}
 \mm(\Psi,\psi)=\mb(\psi), \quad {\rm{for}}\  {\rm{all}} \ \psi\in\{\psi\in H^1(E_1): \psi=0 \quad {\rm{on}}\quad  \Gamma_{L^\ast,\zeta}\},
 \end{equation}
 where
  \begin{equation*}
  \begin{aligned}
  &\mm(\Psi,\psi)=\int_{E_1}\n\Psi\n\psi\de \bm \zeta,\\
  & \mb(\psi)=-\int_{E_1}\mf_1\psi\de \bm \zeta
   +\int_{E_2}\e\mf_3\psi\de  \zeta_2\de  \zeta_3+\int_{0}^{L^\ast}\frac12\mf_4(s)\psi(s,\frac12)\de s.
   \end{aligned}
   \end{equation*}
   By the Lax-Milgram theorem, there exists a unique weak solution $ \Psi\in  H^1(E) $. Then multiplying $ \Psi $ on the sides  of the equation \eqref{4-38} and integrating over $ E_1 $ yield that
    \begin{equation}\label{4-40}
    \|\n \Psi\|_{L^2(E_1)}^2\leq C\| \mf_1\|_{0,\alpha;E_1}^{({1-\alpha;\Gamma_{w,\zeta}})} \| \Psi\|_{L^2(E_1)}+C\| \e\mf_3\|_{L^2(E_2)} \| \Psi\|_{L^2(E_2)}+C\| \mf_4\|_{L^2[0,L^\ast]} \| \Psi\|_{{L^2[0,L^\ast]}}.
    \end{equation}
   Then it follows from Poincar\'{e} inequality and the trace theorem that one obtains
    \begin{equation}\label{4-41}
    \|\Psi\|_{H^1(E_1)}\leq C\left(\|\mf_1\|_{0,\alpha;E_1}^{({1-\alpha;\Gamma_{w,\zeta}})}
    +\| \e\mf_3\|_{1,\alpha;E_2}^{(-\alpha,\Gamma_{w,\zeta})}+\| \mf_4\|_{0,\alpha;[0,L^\ast]}\right).
    \end{equation}
    \par  Next, we can follow the analogous argument as in Step 2 to obtain $ C^{0,\alpha} $ estimate for $ \Psi $.  Then  the Schauder estimate in Theorem 4.6 of \cite{LG13}
 implies that
     \begin{equation}\label{4-42}
 \|\Psi\|_{2,\alpha;E_1}^{(-1-\alpha,\Gamma_{w,\zeta})} \leq  C\left(\|\mf_1\|_{0,\alpha;E_1}^{({1-\alpha;\Gamma_{w,\zeta}})}
    +\| \e\mf_3\|_{1,\alpha;E_2}^{(-\alpha,\Gamma_{w,\zeta})}+\| \mf_4\|_{0,\alpha;[0,L^\ast]}\right).
  \end{equation}
Finally, it follows from the definition of $ \phi_2 $ that
  \begin{equation}\label{4-43}
  \sum_{j=1}^2 \|K_j\|_{1,\alpha;\Omega^\ast}^{(-\alpha,\Sigma^\ast)}
  \leq C
  \left(\|\mf_1\|_{0,\alpha;\Omega^\ast}^{({1-\alpha;\Sigma^\ast})}
    +\| \e\mf_3\|_{1,\alpha;[0,\frac12)}^{(-\alpha;\{\frac12\})}+\| \mf_4\|_{0,\alpha;[0,L^\ast]}\right).
  \end{equation}
  \par  { \bf Step 4}: By recalling the transformation \eqref{4-8} and combining the estimates \eqref{4-35} and \eqref{4-43}, we conclude that the boundary value problem \eqref{4-6} has a unique solution $ (W_1,W_2) \in \left( C_{1,\alpha}^{(-\alpha,\Sigma)}(\Omega)\right)^2 $ satisfying
  \begin{equation*}
  \sum_{j=1}^2 \|W_j\|_{1,\alpha;\Omega}^{(-\alpha,\Sigma)}\leq
  C\left(\sum_{j=1}^2 \|\mf_j\|_{0,\alpha;\Omega}^{(1-\alpha,\Sigma)}+
    \|\mf_3\|_{1,\alpha;[0,\frac12)}^{(-\alpha;\{\frac12\})}+ \|\mf_4\|_{0,\alpha;[0,L]}\right).
    \end{equation*}
    \par Thus, the proof of Lemma 4.1 is completed.
 \end{proof}
 \subsection{Solving the nonlinear boundary value problem  }\noindent
 \par   For a given $\h{{\bf W}}\in \mj(\delta_1)$,
  it follows from Lemma  4.1 that the problem \eqref{4-6} has a unique solution $ (W_1,W_2) \in \left( C_{1,\alpha}^{(-\alpha,\Sigma)}(\Omega)\right)^2 $ satisfying the estimate  \eqref{4-7}.   Define a map $ \mt $ as follows
\begin{equation}\label{4-44}
\mt( \h{{\bf W}})=({{\bf W}}),
\quad{\rm{ for }} \  \h{{\bf W}}\in \mj(\delta_1).
\end{equation}
  The estimate   \eqref{4-7}, together with  \eqref{4-a}, \eqref{4-d} and \eqref{4-5}, yields
    \begin{equation}\label{4-45}
    \begin{aligned}
   \sum_{j=1}^3 \|W_j\|_{1,\alpha;\Omega}^{(-\alpha,\Sigma)}
    \leq
  \mc_1
  \left(\delta_1^2+\delta_1\|w\|_{0,\alpha;[0,L]}
   +\|w\|_{0,\alpha;[0,L]}
+\sigma_{cd}\right),
\end{aligned}
 \end{equation}
 where $ \mc_1>0 $ depends only on  $ (\bm {U}_b^-,L,\alpha) $.
\par  We assume that
\begin{equation}\label{4-46}
\|w\|_{0,\alpha;[0,L] }\leq \delta_2,
\end{equation}
where $ \mc_1 \delta_2\leq \frac{\delta_1}{4}  $
with $ \mc_1 $ given in \eqref{4-45}.  Let $\sigma_2=\frac{1}{4(\mc_1^2+\mc_1)}$ and choose $\delta_1= 4\mc_1\sigma_{cd}$. Then if $ \sigma_{cd}\leq \sigma_2 $, one has
\begin{equation}\label{4-47}
    \sum_{j=1}^3 \|W_j\|_{1,\alpha;\Omega}^{(-\alpha,\Sigma)}
    \leq \delta_1.
  \end{equation}
  Hence $ \mt $ maps $ \mj(\delta_1) $ into itself.
  \par In the following, we will show that $\mt$ is a contraction in $\mj(\delta_1)$. Let $\h{{\bf W}}^k\in \mj(\delta_1), k=1,2$, one has ${\bf W}^k= \mt(\h{{\bf W}}^k)$ for $k=1,2$. Define
\begin{equation*}
{\bf Y}= {\bf W}^1-{\bf W}^2,\quad\quad \h{{\bf Y}}= \h{{\bf W}}^1-\h{{\bf W}}^2.
\end{equation*}
Then it follows from \eqref{4-2} that
\begin{equation}\label{4-e}
  Y_3(y_1,y_2)={\h \Lambda(y_2)}\left(\frac{1}{ \h r^1(y_1,y_2)}-\frac{1}{ \h r^2(y_1,y_2)}\right),
   \end{equation}
  where
  \begin{equation*}
   \h r^i(y_1,y_2)=\left(2\int_{0}^{y_2}\frac{2s}{ \h \rho(\h W_1^i+u_b^-,\h W_2^i,\h W_3^i, \e A_{0},\e B_{0})(\h W_1^i+u_b^-)(y_1,s)}\de s\right)^{\frac{1}{2}}.
  \end{equation*}
Next, we obtain that $ (Y_1,Y_2) $ satisfies
\begin{equation}\label{4-48}
\begin{cases}
\begin{aligned}
&(1-(M_b^-)^2)\p_{y_1}Y_1
+\p_{y_2}Y_2
+\frac{1}{y_2}Y_2=
\mf_1(\h{{\bf W}}^1,\n{\h{\bf W}}^1,\h A,\h B)-\mf_1(\h{{\bf W}}^2,\n{\h{\bf W}}^2,\h A,\h B), \\
&\p_{y_1}Y_2-\p_{y_2}Y_1
=\mf_2(\h{{\bf W}}^1,\n{\h{\bf W}}^1,\h A,\h B)-\mf_2(\h{{\bf W}}^2,\n{\h{\bf W}}^2,\h A,\h B), \\
 &Y_1(0,y_2)=(\mf_3(\h{{\bf W}}^1,\h J,\h A,\h B)-\mf_3(\h{{\bf W}}^2,\h J,\h A,\h B))(0,y_2),\\
&Y_2(L,y_2)=0,\\
&Y_2(y_1,\frac12)=(\mf_4(\h{{\bf W}}^1,g_{cd})-\mf_4(\h{{\bf W}}^2,g_{cd}))(y_1,\frac12),\\
&Y_2(y_1,0)=0.\\
\end{aligned}
\end{cases}
\end{equation}
where  $\mf_j(\h{{\bf W}}^k), j=1,2,3,4, k=1,2,3 $ are functions defined in \eqref{4-6} by replacing $\h{{\bf W}}$  with $\h{{\bf W}}^k$  respectively.
Then  it follows from  Lemma 4.1 that one can derive
\begin{equation}\label{4-49}
\begin{aligned}
  \sum_{j=1}^3 \|Y_j\|_{1,\alpha;\Omega}^{(-\alpha,\Sigma)}&\leq
  \mc_2\left(\sum_{j=1}^2 \|\mf_j(\h{{\bf W}}^1,\n{\h{\bf W}}^1,\h A,\h B)-\mf_j(\h{{\bf W}}^2,\n{\h{\bf W}}^2,\h A,\h B)\|_{0,\alpha;\Omega}^{(1-\alpha,\Sigma)}\right.\\
 &\qquad\quad\left. +
    \|\mf_3(\h{{\bf W}}^1,\h J,\h A,\h B)-\mf_3(\h{{\bf W}}^2,\h J,\h A,\h B)\|_{1,\alpha;[0,\frac12)}^{(-\alpha;\{\frac12\})}\right.\\
 &\qquad\quad \left.   + \|\mf_4(\h{{\bf W}}^1,g_{cd})-\mf_4(\h{{\bf W}}^2,g_{cd})\|_{0,\alpha;[0,L]}+\left\|{\h \Lambda}\left(\frac{1}{ \h r^1}-\frac{1}{ \h r^2}\right)\right\|_{1,\alpha;\Omega}^{(-\alpha,\Sigma)}
   \right)\\
&\leq \mc_2(\delta_1+\delta_2+\sigma_{cd}) \sum_{j=1}^3 \|\h Y_j\|_{1,\alpha;\Omega}^{(-\alpha,\Sigma)},
    \end{aligned}
    \end{equation}
    where $ \mc_2>0 $ depends only on $ (\bm {U}_b^-,L,\alpha) $.
    Setting
\begin{equation}\label{4-50}
 \sigma_3=\min\left\{\sigma_2,\frac{1}{
4\mc_2(4\mc_1+2)} \right\}.
\end{equation}
Then for $ \sigma_{cd}\leq \sigma_3 $, one has
$ \mc_2 (\delta_1+\delta_2+\sigma_{cd})\leq\mc_2 (4\mc_1+2)\sigma_{cd}\leq \frac{1}{4}  $. Hence the mapping $ \mt $ is a contraction mapping so that $\mt $ has a unique fixed point in $ \mj(\delta_1)$.
\section{The construction of the contact discontinuity surface}\noindent
  \par Up to now, for a given function $ g_{cd}(y_1)=\int_{0}^{y_1}w(s)\de s+\frac12  $ satisfying $g_{cd}^\prime(L)=0 $,  we have obtained the solution $ (u_x,u_r,u_\theta) $ for the nonlinear  boundary value problem \eqref{3-20}-\eqref{3-21}. To complete the proof of Theorem 2.3, we will use the implicit function theorem to find the contact discontinuity  $ g_{cd}(y_1) $ such that \eqref{3-22} is satisfied.
  \par  First,
  define a Banach space
\begin{equation*}
 \ml=\{w: \|w \|_{0,\alpha;[0,L] }
  < \infty\}.
  \end{equation*}
    Set
  \begin{equation}\label{5-1}
 \ml_1=\{ w:w(L)=0, \|w \|_{0,\alpha;[0,L] }
 < \infty\}
  \end{equation}
  and
  \begin{equation}\label{5-a}
 \ml_1(\delta_2)=\{w\in \ml_1:  \|w \|_{0,\alpha;[0,L] }
 \leq \delta_2\},
  \end{equation}
  where  $ \delta_2$ is defined in \eqref{4-46}. Then for any $ w\in  \ml_1(\delta_2)$,
   the nonlinear boundary value problem \eqref{3-20}-\eqref{3-21} has a unique solution $ (u_x,  u_r,u_\theta) $ satisfying
 \begin{equation}\label{5-2}
   \|u_x-u_b^-\|_{1,\alpha;\Omega}^{(-\alpha,\Sigma)}+
   \|u_r\|_{1,\alpha;\Omega}^{(-\alpha,\Sigma)}+ \|u_\theta\|_{1,\alpha;\Omega}^{(-\alpha,\Sigma)}\leq 4\mc_1\sigma_{cd}.
   \end{equation}
   \par Let
\begin{equation*}
   \begin{aligned}
  \ml_0= C^{1,\alpha}([0,1/2])\times C^{1,\alpha}([0,1/2])\times C^{1,\alpha}([0,1/2])\times
    C^{1,\alpha}([0,1/2]).
   \end{aligned}
    \end{equation*}
    Then we set
    \begin{equation}\label{5-3}
    \ml_2(\delta_3)=\{
    \bm\varphi_0\in \ml_0:\|\bm \varphi_0-\bm \varphi_b\|_{ \ml_0}\leq \delta_3\},
    \end{equation}
    where \begin{equation*}
     \bm \varphi_0=( J_{0}, \nu_{0},A_0, B_0) \quad {\rm{and}} \quad
     \bm \varphi_b= (J_{b}^-, 0, A_b^-, B_b^-).
   \end{equation*}
   \par Define a map $ \mq:  \ml_2(\delta_3) \times  \ml_1(\delta_2)\rightarrow   \ml$ by
  \begin{equation}\label{5-4}
  \mq(\bm \varphi_0,w):=N(W_1,W_2,W_3, \e A_{0},\e B_{0}) (y_1,\frac12)-P_b,
  \end{equation}
  where
  \begin{equation*}
  N(W_1,W_2,W_3, \e A_{0},\e B_{0})(y_1,\frac12)=\e A_{0}\left(\frac12\right)\left(\rho(W_1+u_b^-,W_2,W_3, \e A_{0},\e B_{0})\right)^\gamma(y_1,\frac12).
   \end{equation*}
  Hence \eqref{3-22} can be written as the equation
  \begin{equation}\label{5-5}
  \mq(\bm \varphi_0,w)=0,
  \end{equation}
   which will be solved \eqref{5-5} by employing the implicit function theorem. For the precise statement of the implicit function theorem, one can see Theorem 3.3 in \cite{WZ23}.  We will verify the conditions $ \rm(i) $, $ \rm(ii) $ and $ \rm(iii) $ in Theorem 3.3 of \cite{WZ23}.
      \par  Obviously,
  \begin{equation*}
  \mq(\bm \varphi_b,0)=0.
  \end{equation*}
  Next, the proof is divided  into two steps.
  \par {\bf Step 1. Differentiability of $ \mq $.}
   \par
   Given any $ w\in\ml_1(\delta_2),w_1\in \ml_1  $, and  $ \tau>0 $, let $( u_x, u_r,  u_\th)=( W_1+u_b^-,W_2, W_3) $ be the solution of \eqref{3-20} with the following boundary conditions:
     \begin{equation*}
\begin{cases}
\rho  u_x(0,y_2 )=\e J_0(y_2), \ u_\theta(0,y_2)=\e \nu_{0}(y_2),\quad &{\rm{on}}\quad \Sigma_0,\\
 u_r(L,y_2)=0, \quad &{\rm{on}}\quad \Sigma_L,\\
 u_r(y_1,\frac12)=\e u_x(y_1,\frac12)w(y_1),\quad &{\rm{on}}\quad \Sigma,\\
 u_r(y_1,0)=0, \quad &{\rm{on}}\quad \Sigma_a,\\
\end{cases}
\end{equation*}
   and $(\e u_x, \e u_r, \e u_\th)=(\e W_1+u_b^-,\e W_2, \e  W_3) $ be the solution of \eqref{3-20} with the following boundary conditions:
   \begin{equation*}
\begin{cases}
\e\rho \e u_x(0,y_2 )=\e J_0(y_2), \ \e u_\theta(0,y_2)=\e \nu_{0}(y_2),\quad &{\rm{on}}\quad \Sigma_0,\\
\e u_r(L,y_2)=0, \quad &{\rm{on}}\quad \Sigma_L,\\
\e u_r(y_1,\frac12)=\e u_x(y_1,\frac12)(w+\tau w_1)(y_1),\quad &{\rm{on}}\quad \Sigma,\\
\e u_r(y_1,0)=0, \quad &{\rm{on}}\quad \Sigma_a.\\
\end{cases}
\end{equation*}
Then it follows from Section 4 that
   \begin{equation}\label{5-6}
   \e W_3=\mf_0(\e{\bf W},\h\Lambda,\h A,\h B), \quad {\rm{and}}\quad \ W_3=\mf_0({\bf W},\h\Lambda,\h A,\h B),
    \end{equation}
    where
     \begin{equation*}
     \begin{aligned}
      \mf_0(\e{\bf W},\h\Lambda,\h A,\h B):=\frac{\h \Lambda}{\e r}=\frac{\h \Lambda}{\left(2\int_{0}^{y_2}\frac{2s}{ \rho(\e W_1+u_b^-,\e W_2,\e W_3, \e A_{0},\e B_{0})(\e W_1+u_b^-)(y_1,s)}\de s\right)^{\frac{1}{2}}}, \\
      \mf_0({\bf W},\h\Lambda,\h A,\h B):=\frac{\h \Lambda}{ r}=\frac{\h \Lambda}{\left(2\int_{0}^{y_2}\frac{2s}{ \rho( W_1+u_b^-, W_2, W_3, \e A_{0},\e B_{0})( W_1+u_b^-)(y_1,s)}\de s\right)^{\frac{1}{2}}}.
        \end{aligned}
        \end{equation*}
   Furthermore, $  ( W_1,   W_2) $ and $ (\e W_1,  \e W_2) $  satisfy
  \begin{equation}\label{5-8}
\begin{cases}
\begin{aligned}
&(1-(M_b^-)^2)\p_{y_1}W_1
+\p_{y_2}W_2
+\frac{1}{y_2}W_2=
\mf_1({{\bf W}},\n{{\bf W}},\h A,\h B), \\
&\p_{y_1}W_2-\p_{y_2}W_1
=\mf_2({{\bf W}},\n{{\bf W}},\h A,\h B), \\
 &W_1(0,y_2)=\mf_3({{\bf W}},\h J,\h A,\h B)(0,y_2),\\
&W_2(L,y_2)=0, \\
&W_2(y_1,\frac12)=\left(u_b^-+W_1(y_1,\frac12)\right)w(y_1),\\
&W_2(y_1,0)=0,\\
\end{aligned}
\end{cases}
\end{equation}
and
 \begin{equation}\label{5-7}
\begin{cases}
\begin{aligned}
&(1-(M_b^-)^2)\p_{y_1}\e W_1
+\p_{y_2}\e W_2
+\frac{1}{y_2}\e W_2=
\mf_1(\e{\bf W},\n{\e{\bf W}},\h A,\h B), \\
&\p_{y_1}\e W_2-\p_{y_2}\e W_1
=\mf_2(\e{{\bf W}},\n{\e{\bf W}},\h A,\h B), \\
 &\e W_1(0,y_2)=\mf_3(\e{{\bf W}},\h J,\h A,\h B)(0,y_2),\\
&\e W_2(L,y_2)=0, \\
&\e W_2(y_1,\frac12)=\left(u_b^-+\e W_1(y_1,\frac12)\right)(w+\tau w_1)(y_1),\\
&\e W_2(y_1,0)=0.\\
\end{aligned}
\end{cases}
\end{equation}
Choosing $ \tau_1>0 $ such that  $ \tau_1\| w_1\|_{0,\alpha;[0,L]}\leq \delta_2 $. Then for $ \tau \in (0,\tau_1) $, it follows from  Section 4  that one gets
\begin{equation}\label{5-b}
\sum_{j=1}^3 \|\e W_j\|_{1,\alpha;\Omega}^{(-\alpha,\Sigma)}
    \leq 4\mc_1\sigma_{cd}.
    \end{equation}
\par Denote
  \begin{equation*}
  {\bf W}^{\tau}=\frac{\e{{\bf W}}-{{\bf W}}}{\tau}.
  \end{equation*}
  It follows from \eqref{5-6}-\eqref{5-8} that one obtains
  \begin{equation}\label{5-9}
\begin{cases}
\begin{aligned}
&W_3^\tau=\frac{\e\mf_0(\e{\bf W},\h\Lambda,\h A,\h B)-\e\mf_0({\bf W},\h\Lambda,\h A,\h B)}{\tau},\\
&(1-(M_b^-)^2)\p_{y_1}W_1^\tau
+\p_{y_2}W_2^\tau
+\frac{1}{y_2}W_2^\tau=
\frac{\mf_1(\e{{\bf W}},\n{\e{\bf W}},\h A,\h B)-\mf_1({{\bf W}},\n{{\bf W}},\h A,\h B)}{\tau}, \\
&\p_{y_1}W_2^\tau-\p_{y_2}W_1^\tau
=\frac{\mf_2(\e{{\bf W}},\n{\e{\bf W}},\h A,\h B)-\mf_2({{\bf W}},\n{{\bf W}},\h A,\h B)}{\tau}, \\
 &W_1^\tau(0,y_2)=\frac{\mf_3(\e{{\bf W}},\h J,\h A,\h B)-\mf_3({{\bf W}},\h J,\h A,\h B)}{\tau}(0,y_2),\\
&W_2^\tau(L,y_2)=0, \\
&W_2^\tau(y_1,\frac12)=\left(u_b^-+\e W_1(y_1,\frac12)\right)w_1(y_1)+W_1^\tau(y_1,\frac12) w(y_1),\\
&W_2^\tau(y_1,0)=0.\\
\end{aligned}
\end{cases}
\end{equation}
Then for $ \tau \in (0,\tau_1) $, one can apply   \eqref{5-b} and \eqref{4-7} to obtain the  following estimate:
 \begin{equation}\label{5-10}
\begin{aligned}
\sum_{j=1}^3 \|W_j^\tau\|_{1,\alpha;\Omega}^{(-\alpha,\Sigma)}&\leq
\mc_3\left\|\frac{\mf_0(\e{{\bf W}},\h\Lambda,\h A,\h B)-\mf_0({{\bf W}},\h\Lambda,\h A,\h B)}{\tau}\right\|_{1,\alpha;\Omega}^{(-\alpha,\Sigma)}\\
 &\quad+ \mc_3\left(\sum_{j=1}^2 \left\|\frac{\mf_j(\e{{\bf W}},\n{\e{\bf W}},\h A,\h B)-\mf_j({{\bf W}},\n{{\bf W}},\h A,\h B)}{\tau}\right\|_{0,\alpha;\Omega}^{(1-\alpha,\Sigma)}\right)\\
 &\quad+\mc_3\left\|\frac{\mf_3(\e{{\bf W}},\h J,\h A,\h B)-\mf_3({{\bf W}},\h J,\h A,\h B)}{\tau}\right\|_{1,\alpha;[0,\frac12)}^{(-\alpha;\{\frac12\})}\\
   &\quad+\mc_3 \|(u_b^-+\e W_1)w_1+W_1^\tau w\|_{0,\alpha;[0,L]}\\
   &\leq \mc_3\left(\delta_2+4\mc_1\sigma_{cd}\right)\sum_{j=1}^3 \|W_j^\tau\|_{1,\alpha;\Omega}^{(-\alpha,\Sigma)}
  +4\mc_3
  \mc_1\sigma_{cd}\| w_1\|_{0,\alpha;[0,L]}\\
  &\quad+\mc_3\| w_1\|_{0,\alpha;[0,L]}\\
  &\leq\mc_3((4\mc_1+1)\sigma_{cd})
  \sum_{j=1}^3 \|W_j^\tau\|_{1,\alpha;\Omega}^{(-\alpha,\Sigma)}
  +\mc_3((4\mc_1+1)\sigma_{cd})\| w_1\|_{0,\alpha;[0,L]}\\
  &\quad+\mc_3\| w_1\|_{0,\alpha;[0,L]},\\
  \end{aligned}
    \end{equation}
     where $ \mc_3>0 $ depends only on   $ (\bm {U}_b^-,L,\alpha) $.
    Setting
  \begin{equation}\label{5-11}
  \sigma_4=\min\left\{\sigma_3,\frac{1}{4\mc_3(1+4\mc_1)}\right\},
   \end{equation}
where $ \sigma_3 $ is defined in \eqref{4-50}. Then for $ \tau \in (0,\tau_1) $ and $ \sigma_{cd}\leq \sigma_4 $, one has
  \begin{equation}\label{5-12}
  \sum_{j=1}^3 \|W_j^\tau\|_{1,\alpha;\Omega}^{(-\alpha,\Sigma)}\leq (2\mc_3+2)
  \| w_1\|_{0,\alpha;[0,L]}.
  \end{equation}
   Hence there exists a subsequence $\{\tau_k\}_{k=1}^\infty $ such that $ (W_1^{\tau_k},W_2^{\tau_k},W_3^{\tau_k}) $ converges to$ (W_1^{0},W_2^{0},W_3^{0}) $  in  $ C_{1,\alpha^\prime}^{(-\alpha^\prime,\Sigma)}(\Omega) $
   as $ \tau_k\rightarrow 0 $ for some $ 0<\alpha^\prime<\alpha $.
The estimate \eqref{5-12} also implies that  $ (W_1^{0},W_2^{0}, W_3^{0})\in \left(C_{1,\alpha}^{(-\alpha,\Sigma)}(\Omega)\right)^3 $ and
  \begin{equation}\label{5-13}
  \sum_{j=1}^3 \|W_j^0\|_{1,\alpha;\Omega}^{(-\alpha,\Sigma)}\leq
   (2\mc_3+2)\| w_1\|_{0,\alpha;[0,L]}.
  \end{equation}
  \par Define a map $ D_w\mq(\bm\varphi_0,w) $ by
\begin{equation}\label{5-14}
  \begin{aligned}
 D_w\mq(\bm\varphi_0,w)(w_1)
  =\sum_{j=1}^3\p_{W_j}N(W_1,W_2,W_3, \e A_{0},\e B_{0})W_j^0
  (y_1,\frac12). \
  \end{aligned}
  \end{equation}
   Then \eqref{5-13}  implies that $ D_w\mq(\bm\varphi_0,w) $ is  a linear mapping from $  \ml_1$ to $  \ml $.   Next, we need to show that $ D_w\mq(\bm\varphi_0,w) $  is  the Fr\'{e}chet derivative of the functional $\mq(\bm\varphi_0,w)$ with respect to $ w $.
    To this end,
   It follows from \eqref{5-9}  that ${\bf W}^{0}$ satisfies
  \begin{equation}\label{5-15}
  \begin{cases}
  \begin{aligned}
 & W_3^0=\sum_{j=1}^3\p_{W_j}\mf_0({{\bf W}},\h\Lambda,\h A,\h B)W_j^0,\\
  &(1-(M_b^-)^2)\p_{y_1}W_1^0
+\p_{y_2}W_2^0
+\frac{1}{y_2}W_2^0\\
&=\sum_{j=1}^3\left(\p_{W_j}\mf_1({{\bf W}},\n{{\bf W}},\h A,\h B)W_j^0+\p_{\n W_j}\mf_1({{\bf W}},\n{{\bf W}},\h A,\h B)\n W_j^0\right), \\
&\p_{y_1}W_2^0-\p_{y_2}W_1^0\\
&=\sum_{j=1}^3\left(\p_{W_j}\mf_2({{\bf W}},\n{{\bf W}},\h A,\h B)W_j^0+\p_{\n W_j}\mf_2({{\bf W}},\n{{\bf W}},\h A,\h B)\n W_j^0\right), \\
 &W_1^0(0,y_2)=\left(\sum_{j=1}^3\p_{W_j}\mf_3({{\bf W}},\h J,\h A,\h B)W_j^0\right)(0,y_2),\\
&W_2^0(L,y_2)=0, \
W_2^0(y_1,\frac12)=\left(u_b^-+ W_1(y_1,\frac12)\right)w_1(y_1)+W_1^0(y_1,\frac12) w(y_1),\
W_2^0(y_1,0)=0.\\
    \end{aligned}
    \end{cases}
  \end{equation}
   By taking difference  of \eqref{5-9} and \eqref{5-15},
    the following estimate can be derived:
  \begin{equation}\label{5-17}
  \begin{aligned}
    &\sum_{j=1}^3 \|W_j^\tau-W_j^0\|_{1,\alpha;\Omega}^{(-\alpha,\Sigma)}\\
  &\leq \mc_4\left\|\frac{\mf_0(\e{\bf W},\h\Lambda,\h A,\h B)-\mf_0({\bf W},\h\Lambda,\h A,\h B)}{\tau}-\sum_{j=1}^3\p_{W_j}\mf_0({{\bf W}},\h\Lambda,\h A,\h B)W_j^0\right\|_{1,\alpha;\Omega}^{(-\alpha,\Sigma)}\\
  &\quad+\mc_4\sum_{i=1}^2\left\|\frac{\mf_i(\e{{\bf W}},\n{\e{{\bf W}}},\h A,\h B)-\mf_i({{\bf W}},\n{{\bf W}},\h A,\h B)}{\tau}\right.\\
  &\qquad\qquad\quad\left.-\sum_{j=1}^3\left(\p_{W_j}\mf_i({{\bf W}},\n{{\bf W}},\h A,\h B)W_j^0+\p_{\n W_j}\mf_i({{\bf W}},\n{{\bf W}},\h A,\h B)\n W_j^0\right)\right\|_{0,\alpha;\Omega}^{(1-\alpha,\Sigma)}\\
   &\quad+\mc_4 \left\|\frac{\mf_3(\e{{\bf W}},\h J,\h A,\h B)-\mf_3({{\bf W}},\h J,\h A,\h B)}{\tau}-\sum_{j=1}^3\p_{W_j}\mf_3({{\bf W}},\h J,\h A,\h B)W_j^0\right\|_{1,\alpha;[0,\frac12)}^{(-\alpha;\{\frac12\})}\\
 &\quad+\mc_4\|\tau W_1^\tau w_1+(W_1^\tau-W_1^0)w\|_{0,\alpha;[0,L]}\\
 &\leq \mc_4\left(\delta_2
  +4\mc_1\sigma_{cd}\right)\sum_{j=1}^3 \|W_j^\tau-W_j^0\|_{1,\alpha;\Omega}^{(-\alpha,\Sigma)}
   +\mc_4(2\mc_3+2) \tau (\| w_1\|_{0,\alpha;[0,L]})^2\\
   &\leq \mc_4\left((4\mc_1+1)\sigma_{cd}
  \right)\sum_{j=1}^3 \|W_j^\tau-W_j^0\|_{1,\alpha;\Omega}^{(-\alpha,\Sigma)}
 +\mc_4(2\mc_3+2) \tau (\| w_1\|_{0,\alpha;[0,L]})^2,
\end{aligned}
    \end{equation}
    where   $ \mc_4>0 $   depends only on  $ (\bm {U}_b^-,L,\alpha) $.  Setting
  \begin{equation}\label{5-18}
   \sigma_5=\min\left\{\sigma_4,\frac{1}{8\mc_4(1+4\mc_1)}\right\}.
   \end{equation}
   Then for $ \sigma_{cd}\leq \sigma_5 $, one has
\begin{equation}\label{5-19}
  \sum_{j=1}^3 \|W_j^\tau-W_j^0\|_{1,\alpha;\Omega}^{(-\alpha,\Sigma)}\leq 4\mc_4(2\mc_3+2)
  \tau \left(\| w_1\|_{0,\alpha;[0,L]}\right)^2.
  \end{equation}
  \par By the definition of $ \mq(\bm\varphi_0,w) $ and $ D_w\mq(\bm \varphi_0,w) $, one gets
  \begin{equation}\label{5-20}
  \begin{aligned}
  &\left\| \mq(\bm\varphi_0,w+\tau w_1)-\mq(\bm\varphi_0,w)- D_w\mq(\bm\varphi_0,w)( \tau w_1)\right\|_{0,\alpha;[0,L]}\\
  &=\left\|N(\e W_1,\e W_2,\e W_3, \e A_{0},\e B_{0})(y_1,\frac12)-N(W_1,W_2,W_3, \e A_{0},\e B_{0})(y_1,\frac12)\right.\\
  &\quad\quad-\left.\tau\left(\sum_{j=1}^3\p_{W_j}N(W_1,W_2,W_3, \e A_{0},\e B_{0}) W_j^0\right)(y_1,\frac12)\right\|_{0,\alpha;[0,L]}\\
  &\leq \left\|N(\e W_1,\e W_2,\e W_3, \e A_{0},\e B_{0})(y_1,\frac12)-N(W_1,W_2,W_3, \e A_{0},\e B_{0})(y_1,\frac12)\right.\\
  &\quad\quad-\left.\tau\left(\sum_{j=1}^3\p_{W_j}N(W_1,W_2,W_3, \e A_{0},\e B_{0}) W_j^\tau\right)(y_1,\frac12)\right\|_{0,\alpha;[0,L]}\\
  &\quad+\left\|\tau\left(\sum_{j=1}^3\p_{W_j}N(W_1,W_2,W_3, \e A_{0},\e B_{0}) (W_j^\tau-W_j^0)\right)(y_1,\frac12)\right\|_{0,\alpha;[0,L]}\\
  &\leq \mc\tau^2\left(\sum_{j=1}^3 \|W_j^\tau\|_{1,\alpha;\Omega}^{(-\alpha,\Sigma)}\right)^2+
  \mc\tau\sum_{j=1}^3 \|W_j^\tau-W_j^0\|_{1,\alpha;\Omega}^{(-\alpha,\Sigma)}\\
  &\leq \mc \tau^2\left(\| w_1\|_{0,\alpha;[0,L]}\right)^2,
  \end{aligned}
    \end{equation}
    where   $ \mc>0 $   depends only on  $ (\bm {U}_b^-,L,\alpha) $.  Thus it holds that
    \begin{equation*}
  \frac{\left\| \mq(\bm\varphi_0,w+\tau w_1)-\mq(\bm\varphi_0,w)- D_w\mq(\bm\varphi_0,w)( \tau w_1)\right\|_{0,\alpha;[0,L]}}
  {\tau\|w_1\|_{0,\alpha;[0,L]}}\rightarrow 0
  \end{equation*}
  as $ \tau\rightarrow 0 $. Hence $ D_w\mq(\bm\varphi_0,w) $ is   the Fr\'{e}chet derivative of the functional $ \mq(\bm\varphi_0,w)$ with respect to $ w $.
  \par It remains to prove  the continuity of the map $ \mq(\bm\varphi_0,w) $ and $ D_w\mq(\bm\varphi_0,w) $. For any fixed $(\bm \varphi_0, w)\in \ml_0\times \ml_1(\delta_2) $, we assume that $ (\bm \varphi_0^k,w^k)\rightarrow (\bm \varphi_0, w) $ in $ \ml_0\times \ml_1(\delta_2)$ as $ k\rightarrow \infty $. Then we first show that as $ k\rightarrow \infty $,
  \begin{equation}\label{5-21}
  \mq(\bm\varphi_0^k,w^k)\rightarrow  \mq(\bm\varphi_0,w), \quad {\rm{in}}\quad  \ml.
  \end{equation}
  By \eqref{5-6} and \eqref{5-8}, ${\bf W}^{k}$ satisfies the following problem:
\begin{equation}\label{5-22}
\begin{cases}
\begin{aligned}
&W_3^k=\mf_0({{\bf W}}^k,\h\Lambda^k,\h A^k,\h B^k),\\
&(1-(M_b^-)^2)\p_{y_1}W_1^k
+\p_{y_2}W_2^k
+\frac{1}{y_2}W_2^k=
\mf_1({{\bf W}}^k,\n{{\bf W}}^k,\h A^k,\h B^k), \\
&\p_{y_1}W_2^k-\p_{y_2}W_1^k
=\mf_2({{\bf W}}^k,\n{{\bf W}}^k,\h A^k,\h B^k), \\
 &W_1^k(0,y_2)=\mf_3({{\bf W}}^k,\h J^k,\h A^k,\h B^k)(0,y_2), \\
&W_2^k(L,y_2)=0, \\
&W_2^k(y_1,\frac12)=\left(u_b^-+W_1^k(y_1,\frac12)\right)w^k(y_1),\\
&W_2^k(y_1,0)=0,\\
\end{aligned}
\end{cases}
\end{equation}
where
\begin{equation*}
(\h J^k,\h\Lambda^k,\h A^k,\h B^k)
=(\e J_0^k,\e\nu_0^k,\e A_0^k,\e B_0^k)-(J_b^-,0,A_b^-,B_b^-).
\end{equation*}
Taking the difference of \eqref{5-8} and  \eqref{5-22}, one obtains that
  \begin{equation}\label{5-23}
\begin{cases}
\begin{aligned}
&W_3^k-W_3^0=\mf_0({{\bf W}}^k,\h\Lambda^k,\h A^k,\h B^k)-\mf_0({{\bf W}},\h\Lambda,\h A,\h B),\\
&(1-(M_b^-)^2)\p_{y_1}(W_1^k-W_1)
+\p_{y_2}(W_2^k-W_2)
+\frac{1}{y_2}(W_2^k-W_2)\\
&=
\mf_1({{\bf W}}^k,\n{{\bf W}}^k,\h A^k,\h B^k)-\mf_1({{\bf W}},\n{{\bf W}},\h A,\h B), \\
&\p_{y_1}(W_2^k-W_2)-\p_{y_2}(W_1^k-W_1)\\
&=\mf_2({{\bf W}}^k,\n{{\bf W}}^k,\h A^k,\h B^k)-\mf_2({{\bf W}},\n{{\bf W}},\h A,\h B)
, \\
 &(W_1^k-W_1)(0,y_2)=\left(\mf_3({{\bf W}}^k,\h J^k,\h A^k,\h B^k)-\mf_3({{\bf W}},\h J,\h A,\h B)\right)(0,y_2), \\
&(W_2^k-W_2)(L,y_2)=0, \\
&(W_2^k-W_2)(y_1,\frac12)=u_b^-(w^k-w)(y_1)+W_1^k(y_1,\frac12)w^k(y_1)
-W_1(y_1,\frac12)w(y_1),\\
&(W_2^k-W_2)(y_1,0)=0.\\
\end{aligned}
\end{cases}
\end{equation}
By similar   estimate in \eqref{5-10}, one can infer  that
\begin{equation}\label{5-24}
\sum_{j=1}^3 \|W_j^k-W_j\|_{1,\alpha;\Omega}^{(-\alpha,\Sigma)}
\leq \mc \left(\|\bm\varphi_0^k-\bm\varphi_0\|_{\ml_0}+\|w^k- {w}\|_{0,\alpha;[0,L]}\right).
\end{equation}
Therefore,
  \begin{equation}\label{5-25}
  \begin{aligned}
   &\|\mq(\bm\varphi_0^k,w^k)-\mq(\bm\varphi_0,w)\|
   _{0,\alpha;[0,L]}\\
  &=\left\|N(W_1^k,W_2^k,W_3^k, \e A_{0}^k,\e B_{0}^k)-N(W_1,W_2,W_3, \e A_{0},\e B_{0})\right\|
   _{0,\alpha;[0,L]}\\
  &\leq \mc\left(\sum_{j=1}^3 \|W_j^k-W_j\|_{1,\alpha;\Omega}^{(-\alpha,\Sigma)}
+\|\bm\varphi_0^k-\bm\varphi_0\|_{\ml_0}\right)\\
&\leq \mc \left(\|\bm\varphi_0^k-\bm\varphi_0\|_{\ml_0}+\|w^k- {w}\|_{0,\alpha;[0,L]}\right),
  \end{aligned}
  \end{equation}
  which yields \eqref{5-21}.
   \par Next, we  prove the continuity of the map $ D_w\mq(\bm\varphi_0,w) $, i.e. to show that as $ k\rightarrow \infty $,
  \begin{equation}\label{5-26}
  D_w\mq(\bm\varphi_0^k,w^k)\rightarrow   D_w\mq(\bm\varphi_0,w), \quad {\rm{in}}\quad  \ml.
  \end{equation}
  It follows from \eqref{5-15} that
  ${\bf W}^{0,k} $ satisfies the following problem:
  \begin{equation}\label{5-27}
  \begin{cases}
  \begin{aligned}
  & W_3^{0,k}=\sum_{j=1}^3\p_{W_j^k}\mf_0({{\bf W}}^{k},\h\Lambda^{k},\h A^{k},\h B^{k})W_j^{0,k},\\
  &(1-(M_b^-)^2)\p_{y_1}W_1^{0,k}
+\p_{y_2}W_2^{0,k}
+\frac{1}{y_2}W_2^{0,k}\\
&=
\sum_{j=1}^3\left(\p_{W_j^k}\mf_1({{\bf W}}^{k},\n{{\bf W}}^{k},\h A^{k},\h B^{k})W_j^{0,k}-\p_{\n W_j^k}\mf_1({{\bf W}}^{k},\n{{\bf W}}^{k},\h A^{k},\h B^{k})\n W_j^{0,k}\right), \\
&\p_{y_1}W_2^{0,k}-\p_{y_2}W_1^{0,k}\\
&=
\sum_{j=1}^3\left(\p_{W_j^k}\mf_2({{\bf W}}^{k},\n{{\bf W}}^{k},\h A^{k},\h B^{k})W_j^{0,k}-\p_{\n W_j^k}\mf_2({{\bf W}}^{k},\n{{\bf W}}^{k},\h A^{k},\h B^{k})\n W_j^{0,k}\right), \\
&W_1^{0,k}(0,y_2)=\left(\sum_{j=1}^2\p_{W_j^k}\mf_3({{\bf W}}^{k},\h J^{k},\h A^{k},\h B^{k})W_j^{0,k}\right)(0,y_2),\\
&W_2^{0,k}(L,y_2)=0,\\
&W_2^{0,k}(y_1,\frac12)=\left(u_b^-+ W_1^{k}(y_1,\frac12)\right)w_1(y_1
)+W_1^{0,k}(y_1,\frac12) w^{k}(y_1),\\
&W_2^{0,k}(y_1,0)=0.\\
    \end{aligned}
    \end{cases}
  \end{equation}
  Taking the difference of \eqref{5-15} and \eqref{5-27}, one has
  \begin{equation*}
  \begin{aligned}
\sum_{j=1}^3 \|W_j^{0,k}-W_j^0\|_{1,\alpha;\Omega}^{(-\alpha,\Sigma)}
\leq \mc \left(\sum_{j=1}^3 \|W_j^{k}-W_j\|_{1,\alpha;\Omega}^{(-\alpha,\Sigma)}
+\|\bm\varphi_0^k-\bm\varphi_0\|_{\ml_0}+\|w^k- {w}\|_{0,\alpha;[0,L]}\right).
\end{aligned}
\end{equation*}
 Combining \eqref{5-24} yields that
  \begin{equation}\label{5-28}
  \begin{aligned}
  \sum_{j=1}^3 \|W_j^{0,k}-W_j^0\|_{1,\alpha;\Omega}^{(-\alpha,\Sigma)}
\leq \mc \left(
\|\bm\varphi_0^k-\bm\varphi_0\|_{\ml_0}+\|w^k- {w}\|_{0,\alpha;[0,L]}\right).
  \end{aligned}
  \end{equation}
  Then it holds that
   \begin{equation}\label{5-29}
  \begin{aligned}
  &\|D_w\mq(\bm\varphi_0^k,w^k)(w_1)-   D_w\mq(\bm\varphi_0,w)(w_1)\| _{0,\alpha;[0,L]}\\
  &=\left\|\sum_{j=1}^3\left(\p_{W_j^k}N(W_1^k,W_2^k,W_3^k, \e A_{0}^k,\e B_{0}^k) W_j^{0,k}-\p_{W_j}N(W_1,W_2,W_3, \e A_{0},\e B_{0}) W_j^0\right)\right\|_{0,\alpha;[0,L]}\\
   &\leq \mc\left(\sum_{j=1}^3 \|W_j^{0,k}-W_j^0\|_{1,\alpha;\Omega}^{(-\alpha,\Sigma)}+\sum_{j=1}^3 \|W_j^{k}-W_j\|_{1,\alpha;\Omega}^{(-\alpha,\Sigma)}
+\|\bm\varphi_0^k-\bm\varphi_0\|_{\ml_0}\right)\\
&\leq \mc \left(
\|\bm\varphi_0^k-\bm\varphi_0\|_{\ml_0}+\|w^k- {w}\|_{0,\alpha;[0,L]}\right),
   \end{aligned}
  \end{equation}
    which implies that \eqref{5-26} holds.
     \par In particular, at the background state,
  \begin{equation}\label{5-30}
  \begin{aligned}
D_w\mq(\bm\varphi_b,0)(w_1)
  =-J_b^-W_1^b(y_1,\frac12)=-2W_1^b(y_1,\frac12),
  \end{aligned}
  \end{equation}
  where ${\bf W}^{b} $ satisfies
  \begin{equation}\label{5-31}
  \begin{cases}
  \begin{aligned}
  &(1-(M_b^-)^2)\p_{y_1}W_1^b
+\p_{y_2}W_2^b
+\frac{1}{y_2}W_2^b=0, \\
&\p_{y_1}W_2^b-\p_{y_2}W_1^b=0, \\
&W_3^b=0,\\
 &W_1^b(0,y_2)=0,\\
&W_2^b(L,y_2)=0, \\
&W_2^b(y_1,\frac12)=u_b^-w_1(y_1),\\
&W_2^b(y_1,0)=0.\\
    \end{aligned}
    \end{cases}
  \end{equation}
    Define
   \begin{equation}\label{5-c}
 \ml_b=\{v\in \ml: v(0)=0\}.
  \end{equation}
  Then    $ D_w\mq(\bm\varphi_b,0) $ is  a continuous mapping from $  \ml_1 $ to $  \ml_b \subset\ml $.
  \par {\bf Step 2. The isomorphism of $ D_w\mq(\bm \varphi_b,0)$. }
    \par To prove the isomorphism of $ D_w\mq(\bm \varphi_b,0)$, we need to show that for any given function $  P^\ast\in \ml_b$, there exists a unique $  w^\ast\in \ml_1 $ such that $ D_w\mq(\bm \varphi_b,0)(w^\ast)=P^\ast $, i.e.,
\begin{equation}\label{5-32}
\begin{aligned}
P^\ast(y_1)= -2W_1^\ast(y_1,\frac12).
\end{aligned}
\end{equation}
It follows from \eqref{5-31} that the solution $ (W_1^\ast,W_2^\ast,0) $ satisfies
 \begin{equation}\label{5-33}
\begin{cases}
\begin{aligned}
&(1-(M_b^-)^2)\p_{y_1} W_1^\ast
+\p_{y_2} W_2^\ast
+\frac{1}{y_2} W_2^\ast=
0, \\
&\p_{y_1}W_2^\ast-\p_{y_2} W_1^\ast
=0, \\
 & W_1^\ast(0,y_2)=0, \\
& W_2^\ast(L,y_2)=0, \\
& W_1^\ast(y_1,\frac12)=-\frac{P^\ast(y_1)} { 2},\\
& W_2^\ast(y_1,0)=0.\\
\end{aligned}
\end{cases}
\end{equation}
The second equation in \eqref{5-33} implies that there exists a potential function $ \phi^\ast $ such that
 \begin{equation}\label{5-34}
(\p_{y_1}\phi^\ast,\p_{y_2}\phi^\ast)=(W_1^\ast,W_2^\ast), \quad \phi^\ast(L,0)=0.
\end{equation}
Then \eqref{5-33} can be written as
\begin{equation}\label{5-35}
\begin{cases}
\begin{aligned}
&\p_{y_1}((1-(M_b^-)^2)y_2\p_{y_1}\phi^\ast)
+\p_{y_2}(y_2\p_{y_2}\phi^\ast)
=0, \\
 &\p_{y_1}\phi^\ast(0,y_2)=0,\\
&\phi^\ast(L,y_2)=0, \\
&\phi^\ast(y_1,\frac12)=\int_{L}^{y_1}-\frac{P^\ast(s)} { 2}\de s,\\
&\p_{y_2}\phi^\ast(y_1,0)=0.\\
\end{aligned}
\end{cases}
\end{equation}
\par  To deal with the singularity near $ y_2=0 $, we rewrite the problem \eqref{5-35} by using the cylindrical coordinate transformation again. Define
\begin{equation*}
\lambda_1=y_1,\ \lambda_2=y_2\cos\tau,\ \lambda_3=y_2\sin\tau,\ \tau\in[0,2\pi],
\end{equation*}
and
\begin{equation*}
\begin{aligned}
&D_1=\{(\lambda_1,\lambda_2,\lambda_3):0<\lambda_1<L,\lambda_2^2+\lambda_3^2\leq \frac12\}, \quad D_2= \{(\lambda_2,\lambda_3): \lambda_2^2+\lambda_3^2\leq \frac12\},\\
&\Gamma_{w,\lambda}=[0,L]\times\{(\lambda_2,\lambda_3):
\lambda_2^2+\lambda_3^2= \frac12\},\\
&\Gamma_{0,\lambda}=\{0\}\times\{(\lambda_2,\lambda_3):\lambda_2^2
+\lambda_3^2\leq \frac12\},\quad
\Gamma_{L,\lambda}=\{L\}\times\{(\lambda_2,\lambda_3):
\lambda_2^2+\lambda_3^2\leq \frac12\},\\
&\Psi^\ast(\bm \lambda)= \phi^\ast(\lambda_1,\sqrt{\lambda_2^2+\lambda_3^2})=
\Psi^\ast(\lambda_1,|\lambda^\prime|).
\end{aligned}
\end{equation*}
Then $ \Psi^\ast $ solves the following problem
 \begin{equation}\label{5-36}
\begin{cases}
\begin{aligned}
&(1-(M_b^-)^2)\p_{\lambda_1}^2 \Psi^\ast+\p_{\lambda_2}^2 \Psi^\ast+\p_{\lambda_3}^2 \Psi^\ast
=0, \\
 &\p_{\lambda_1}\Psi^\ast(0,|\lambda'|)=0, \quad &{\rm{on}}\quad \Gamma_{0,\lambda},\\
&\Psi^\ast(L,|\lambda'|)=0, \quad &{\rm{on}}\quad \Gamma_{L,\lambda},\\
&\Psi^\ast(\lambda_1,|\lambda'|)=\int_{L}^{\lambda_1}-\frac{P^\ast(s)} { 2}\de s, \quad &{\rm{on}}\quad \Gamma_{w,\lambda}.\\
\end{aligned}
\end{cases}
\end{equation}
 By similar arguments as in the Step 3 of Lemma 4.1, \eqref{5-36} has a unique solution $ \Psi^\ast\in C_{2,\alpha}^{(-1-\alpha,\Gamma_{w,\lambda})}(D_1) $ satisfying
\begin{equation}\label{5-37}
\|\Psi^\ast\|_{2,\alpha;D_1}^{(-1-\alpha,\Gamma_{w,\lambda})}
\leq \mc \|P^\ast\|_{0,\alpha;[0,L]}.
\end{equation}
Thus
\begin{equation}\label{5-38}
\sum_{i=1}^2\|W_i^\ast\|_{1,\alpha;\Omega}^{(-\alpha,\Sigma)}
\leq \mc \|P^\ast\|_{0,\alpha;[0,L]}.
\end{equation}
\par Set $ w^\ast(y_1)=u_b^- W_2^\ast(y_1,\frac12) $, then \eqref{5-38} shows that $ w^\ast\in \ml_1  $. Hence we have shown there exists a unique $ w^\ast\in \ml_1 $ such that $ D_w\mq(\bm\varphi_b,0)(w^\ast)= P^\ast $. The proof of the isomorphism of $ D_w\mq(\bm\varphi_b,0)$ is completed.
\section{Proof of Theorem 2.3}\noindent
\par Now, by the implicit function theorem, there exist  positive constants $ \sigma_6  $ and $ \mc $ depending only on $ (\bm {U}_b^-,L,\alpha) $ such that for $ \delta_3\leq \sigma_6 $, the equation $  \mq(\bm \varphi_0,w)=0 $ has a unique solution $ w=w(\bm \varphi_0) $ satisfying
\begin{equation}\label{6-1}
\|w\|_{0,\alpha;[0,L]}\leq \mc\|\bm \varphi_0-\bm \varphi_b\|_{\ml_0}= \mc\sigma_{cd}.
\end{equation}
Here $ \delta_3 $ is defined in \eqref{5-3}.  Hence  the contact discontinuity $ g_{cd}(y_1)=\int_{0}^{y_1}w(s) \de s+\frac12 $ is determined  and  \eqref{6-1} implies that
\begin{equation}\label{6-2}
\|g_{cd}-\frac12\|_{1,\alpha ;[0,L]}\leq \mc_5\|\bm \varphi_0-\bm \varphi_b\|_{\ml_0}= \mc_5\sigma_{cd},
\end{equation}
where $ \mc_5>0 $ depends only on $ (\bm {U}_b^-,L,\alpha) $ .
\par We choose $  \sigma_1 $ and  $\mc^\ast $ as
   \begin{equation}\label{6-3}
   \sigma_1=\min\{\sigma_5,\sigma_6\} \quad {\rm{ and}} \quad \mc^\ast=4\mc_1+\mc_5,
   \end{equation}
   where $ \sigma_5 $  is defined in \eqref{5-18} and $ \mc_1 $ is  defined in \eqref{4-45}.
  Then if $ \sigma_{cd}\leq\sigma_1 $,  $ \mathbf{Problem \ 3.1}  $ has  a unique  smooth   subsonic solution $ (u_x,u_r,u_\theta;g_{cd})$ satisfying
 \begin{equation}\label{6-4}
 \|u_x -u_b^-\|_{1,\alpha;\Omega}^{(-\alpha,\Sigma)}
 +\|u_r\|_{1,\alpha;\Omega}^{(-\alpha,\Sigma)}+
 \|u_\theta\|_{1,\alpha;\Omega}^{(-\alpha,\Sigma)}+\| g_{cd} -\frac12 \|_{1,\alpha;[0,L]}\leq \mc^\ast\sigma_{cd}.
 \end{equation}
 Furthermore, one has
 \begin{equation}\label{6-5}
 \|(A,B)-(A_b^-,B_b^-)\|_{1,\alpha;\overline\Omega}
 =\|(\e A_0,\e B_0)-(A_b^-,B_b^-)\|_{1,\alpha;\overline\Omega}\leq \mc^\ast\sigma_{cd}.
 \end{equation}

 \par Since the modified  Lagrangian transformation \eqref{3-4} is invertible, thus the solution transformed back in $ (x,r)$-coordinates solves $ \mathbf{Problem \ 2.2}  $ and the estimates \eqref{6-4} and \eqref{6-5}
 imply that the estimates \eqref{2-19} and \eqref{2-20}  in   Theorem 2.3 hold. Thus the proof of
  Theorem 2.3 is completed.
  \par {\bf Acknowledgement.} Weng is partially supported by National Natural Science Foundation of China  11971307, 12071359, 12221001.

\end{document}